\documentclass[reqno,a4paper, 11pt]{amsart}

\usepackage[a4paper=true,pdfpagelabels]{hyperref}
\usepackage{graphicx}

\usepackage[ansinew]{inputenc}
\usepackage{amsfonts,epsfig}
\usepackage{latexsym}
\usepackage{amsmath}
\usepackage{amssymb}

\newtheorem{theorem}{Theorem}
\newtheorem{lemma}[theorem]{Lemma}
\newtheorem{corollary}[theorem]{Corollary}
\newtheorem{proposition}[theorem]{Proposition}
\newtheorem{question}[theorem]{Question}
\newtheorem{lettertheorem}{Theorem}
\newtheorem{letterlemma}[lettertheorem]{Lemma}

\theoremstyle{definition}

\theoremstyle{remark}
\newtheorem{remark}[theorem]{Remark}

\numberwithin{equation}{section}

\newcommand{\D}{\mathbb{D}}
\newcommand{\DD}{\widehat{\mathcal{D}}}
\newcommand{\Dom}{D_{\omega}}
\newcommand{\N}{\mathbb{N}}
\newcommand{\R}{\mathbb{R}}

\newcommand{\C}{\mathbb{C}}

\newcommand{\e}{\varepsilon}
\newcommand{\ep}{\varepsilon}

\renewcommand{\phi}{\varphi}

\newcommand{\T}{\mathbb{T}}

\newcommand{\wtT}{\widetilde{T}}
\newcommand{\wtz}{\widetilde{z}}

\newcommand{\op}{\mathrm{o}}
\def\wo{\widehat{\om}}

\def\a{\alpha}       \def\b{\beta}        \def\g{\gamma}
\def\d{\delta}       \def\De{{\Delta}}    \def\e{\varepsilon}
\def\la{\lambda}     \def\om{\omega}      
       \def\t{\theta}       
                  \def\z{\zeta}
\def\F{\Phi}                  \def\vp{\varphi}

\def\R{{\mathcal R}}
\def\I{{\mathcal I}}
\def\Inv{{\mathcal Inv}}

\DeclareMathOperator{\supp}{supp}
\renewcommand{\H}{\mathcal{H}}
\newcommand{\SSS}{\mathcal{S}}

\begin{document}

\title[Trace class criteria]{{Trace class criteria for Toeplitz and composition operators on small Bergman spaces}}


\author{Jos\'e \'Angel Pel\'aez}
\address{Departamento de An\'alisis Matem\'atico, Universidad de M\'alaga, Campus de
Teatinos, 29071 M\'alaga, Spain} \email{japelaez@uma.es}

\author{Jouni R\"atty\"a}
\address{University of Eastern Finland, P.O.Box 111, 80101 Joensuu, Finland}
\email{jouni.rattya@uef.fi}

\thanks{This research was supported in part by the Ram\'on y Cajal program
of MICINN (Spain); by Ministerio de Edu\-ca\-ci\'on y Ciencia, Spain, projects
MTM2011-25502 and MTM2011-26538;  by   La Junta de Andaluc{\'i}a, (FQM210) and
(P09-FQM-4468);  by Academy of Finland project no. 268009,  by V\"ais\"al\"a Foundation of Finnish Academy of Science and Letters, and by Faculty of Science and Forestry of University of Eastern Finland project no. 930349.
}

\date{\today}


\begin{abstract}
We characterize the Schatten class Toeplitz operators induced by a positive Borel measure on the unit disc and the reproducing kernel of the Bergman space $A^2_\om$, where $\om$ is a radial weight satisfying the doubling property $\int_r^1\om(s)\,ds\le C\int_{\frac{1+r}{2}}^1\om(s)\,ds$. By using this, we describe the Schatten class composition operators. We also discuss basic properties of composition operators acting from $A^p_\om$ to $A^q_v$.
\end{abstract}

\maketitle



\section{Introduction and main results}

Let $\H(\D)$ denote the space of all analytic functions in the unit disc $\D=\{z:|z|<1\}$. If
$0<p<\infty$ and $\omega$ is a weight, i.e. a nonnegative integrable function on $\D$, the weighted Bergman
space $A^p_\omega$ consists of $f\in\H(\D)$ such that
    $$
    \|f\|_{A^p_\omega}^p=\int_\D|f(z)|^p\omega(z)\,dA(z)<\infty,
    $$
where $dA(z)=\frac{dx\,dy}{\pi}$ is the normalized
Lebesgue area measure on $\D$. As usual, $A^p_\alpha$ denotes the weighted Bergman space induced by
the standard radial weight $(1-|z|^2)^\alpha$.

In this study we consider Bergman spaces $A^p_\om$ induced by weights in the class~$\DD$ of the radial~$\om$ for which $\widehat{\om}(r)=\int_r^1\om(s)\,ds$ satisfies $\widehat{\om}(r)\le C\widehat{\om}(\frac{1+r}{2})$. The standard radial weights admit this doubling property, but $\DD$ contains also weights such that $\widehat{\om}(r)$ decreases to zero much slower than any positive power of $1-r$. It is known that for such $\om$, many finer function theoretic properties of $A^p_\om$ are very different from those of $A^p_\a$ and, in particular, its harmonic analysis is similar to that of Hardy spaces in many aspects~\cite{PelRat,PelRatMathAnn,PelRatproj}.

For any $\om\in\DD$, the norm convergence in $A^2_\om$ implies the uniform convergence on compact subsets, and hence the point
evaluations $L_z$ are bounded linear functionals on $A^2_\om$. Therefore there exist reproducing kernels
$B^\om_{z}\in A^2_\om$, with $\|L_z\|=\|B^\om_{z}\|_{A^2_\om}$, such that
    \begin{displaymath}
    L_z(f)=f(z)=\langle f, B^\om_{z}\rangle_{A^2_\om} =\int_{\D} f(\z)\,\overline{B^\om_{z}(\z)}\,\om(\z)\,dA(\z), \quad f\in A^2_\om.
    \end{displaymath}
These kernels give rise to the
Toeplitz operator
    \begin{equation*}\label{intoper}
    \mathcal{T}_\mu(f)(z)=\int_{\D}f(\z)\,\overline{B^\om_{z}(\z)}\,d\mu(\z),
    \end{equation*}
where~$\mu$ is a positive Borel measure on $\D$, which is the primary object in this study. The operator~$\mathcal{T}_\mu$, associated with the kernel of a standard weighted Bergman space~$A^2_\a$ and a measure $d\mu=\phi dA$, has been extensively studied since the seventies \cite{CobInd73, McSuInd79, ZhuTams87}.
Luecking~\cite{Lu87} was probably one of the first authors to consider~$\mathcal{T}_\mu$ with measures as symbols. He characterized those $\mu$ for which $\mathcal{T}_\mu$ belongs to the Schatten-Von Neumann ideal $\mathcal{S}_p(A^2_\alpha)$.
His approach works also on the Hardy space~$H^2$ and, for a certain range of $p$, on the classical weighted Dirichlet spaces~\cite{Lu87,PauPelJDA13}. Since Toeplitz operators go hand in hand with several other operators, Luecking's method has turned out to be useful in subsequent research on concrete operator theory~\cite{Lu87,Zhu}.

Our main result describes the positive Borel measures $\mu$ such that
$\mathcal{T}_\mu$ belongs to $\SSS^p(A^2_\om)$ when $0<p<\infty$ and $\om\in\DD$.
Throughout the proof we will use the norm given
by the Littlewood-Paley identity
    \begin{equation}\label{LP1}
    \|f\|^2_{A^2_\om}=4\|f'\|^2_{A^2_{\om^\star}}+\om(\D)|f(0)|^2,
    \end{equation}
where $\omega^\star(z)=\int_{|z|}^1\omega(s)\log\frac{s}{|z|}s\,ds$~\cite[Theorem~4.2]{PelRat}. In fact, this relation gives rise to the definition of a family of Dirichlet-type spaces $H_\alpha(\om^\star)$ which form the more general context we will work in. What we will do in Section~\ref{Section:Schatten-Toeplitz} is to characterize Schatten class Toeplitz operators on $H_\alpha(\om^\star)$.

Although we follow Luecking's approach, our context leads to severe technical difficulties in the proof. For instance,
in the original argument one uses the closed formula $(1-\overline{z}\z)^{-(2+\alpha)}$ of the reproducing kernel of $A^2_\a$ to see that it
is essentially constant in each hyperbolically bounded region and to obtain asymptotic $L^p_\beta$-estimates for their derivatives. The general situation for $A^2_\om$ is much more complicated because of the lack of explicit expressions for $B^\om_z$. We will use the method from~\cite{PelRatproj}, that is based on a decomposition norm theorem, to deduce precise $L^p_v$-estimates of the derivatives of $B^\om_z$. Also, the question of when the reproducing kernels of $A^2_\om$ remain essentially constant in hyperbolically bounded regions is delicate because even a little perturbation in the weight, that does not change the space it self, might introduce zeros to the kernel functions \cite{ZeyPams10}. Therefore we will circumvent these and other obstacles in the proof by using different techniques.

We need some notation to state our result. Set $\vp_a(z)=(a-z)/(1-\overline{a}z)$ for $a,z\in\D$.
The pseudohyperbolic distance from $z$ to $w$ is $\varrho(z,w)=|\vp_z(w)|$, and the pseudohyperbolic
disc of center $a\in\D$ and radius $r\in(0,1)$ is denoted by
$\Delta(a,r)=\{z:\varrho(a,z)<r\}$. The polar rectangle associated with an arc $I\subset\T$ is
    $$
    R(I)=\left\{z\in\D:\,\frac{z}{|z|}\in I,\,\,1-\frac{|I|}{2\pi}\le |z|<1-\frac{|I|}{4\pi}\right\}.
    $$
Write $z_I=(1-|I|/2\pi)\xi$, where $\xi\in\T$ is the midpoint of $I$. Let $\Upsilon$ denote the family of all dyadic arcs of $\T$. Every arc $I\in\Upsilon$ is of the form
    $$
    I_{n,k}=\left\{e^{i\theta}:\,\frac{2\pi k}{2^n}\le
    \theta<\frac{2\pi(k+1)}{2^n}\right\},
    $$
where $k=0,1,2,\dots,2^n-1$ and $n=\N\cup\{0\}$.
Then the family $\left\{R(I):\,\,I\in\Upsilon\right\}$ consists of
pairwise disjoint rectangles whose union covers~$\D$. For
$I_j\in\Upsilon\setminus\{I_{0,0}\}$, we will write $z_j=z_{I_{j}}$.
For convenience, we associate
the arc $I_{0,0}$ with the point $1/2$.

\begin{theorem}\label{mainToeplitzbergman}
Let $0<p<\infty$, $\om\in\DD$ and $\mu$ be a positive Borel measure on~$\D$. Then the following assertions are equivalent:
\begin{enumerate}
\item[\rm (a)] $\mathcal{T}_\mu\in\SSS_p(A^2_\om)$;
\item[\rm (b)] $\sum_{R_j\in\Upsilon}
    \left(\frac{\mu(R_j)}{\omega^\star(z_j)}\right)^p<\infty$;
\item[\rm (c)] $z\mapsto\frac{\mu\left(\Delta(z,r)\right)}{\om^\star(z)}$ belongs to $L^p\left(\frac{1}{(1-|z|)^2}\right)$ for some $0<r<1$.
\end{enumerate}
\end{theorem}

We will deduce the case $p\ge 1$ by extending the proof of \cite[Theorem~6.11]{PelRat} to the class $\DD$. The result of Constantin~\cite[Theorem~4.3]{OC}, regarding to Bekoll\'e-Bonami weights, implies the assertion in the case $p\ge 1$ for regular weights. The proof is based on interpolation and does not work for the whole class $\DD$. The most involved part in the proof of Theorem~\ref{mainToeplitzbergman} in the case $p<1$ is to show that~(b) is a necessary condition for $\mathcal{T}_\mu$ to belong to $\SSS_p(A^2_\om)$. One obstacle here is that the condition $\sum_n|\langle T(e_n),e_n\rangle|^p<\infty$ does not describe Schatten class operators.

We will apply Theorem~\ref{mainToeplitzbergman} to characterize Schatten class composition operators on $A^2_\om$.
Each analytic self-map $\vp$ of $\D$ induces the composition operator $C_\vp(f)= f\circ\vp$
acting on $\H(\D)$. With regard to the theory of composition operators, we refer to~\cite{CowenMac95,Shapiro93,Zhu}.
The study of composition operators on spaces of analytic functions has attracted a lot of attention since the mid 60's~\cite{No68,Ry66}.
Since then, the number of papers concerning this topic has grown enormously. Despite this fact, it seems that the existing literature does not offer too much information about the basic properties of composition operators acting on the Bergman spaces $A^p_\om$ when $\om\in\DD$. Therefore, we will study these properties by passing on the way to Schatten classes. Now we proceed to the main findings in this direction and the definitions needed for the statements.

Let $\z\in\vp^{-1}(z)$ denote the set of the points
$\{\z_n\}$ in $\D$, organized by increasing moduli, such that
$\vp(\z_n)=z$ for all $n$, with each point repeated according to its multiplicity.
For a radial weight $\om$ and an analytic self-map $\vp$ of $\D$
we define the generalized Nevanlinna counting function as
    $$
    N_{\vp,\om^\star}(z)=\sum_{\z\in\vp^{-1}(z),}\om^\star\left(z\right),\quad z\in\D\setminus\{\vp(0)\}.
    $$
Here and from now on we set $\om^\star(\zeta)=0$ for $\z\in\T$. A change of variable together with \cite[Theorem~4.2]{PelRat} shows that
$\limsup_{|z|\to1^-}\frac{N_{\vp,\om^\star}(z)}{\om^\star(z)}<\infty$ is a sufficient condition for $C_\vp$ to be bounded on $A^p_\om$.
We will show that the counting function $N_{\vp,\om^\star}$ is subharmonic outside of the image of the origin under $\vp$, and further, in that set it
satisfies the sharp pointwise estimate $N_{\vp,\om^\star}(z)\le \omega^\star\left(\frac{z-\vp(0)}{1-\overline{\vp(0)}z}\right)$.
By using this, we will deduce that $\om\in\DD$ if and only if
$\limsup_{|z|\to1^-}\frac{N_{\vp,\om^\star}(z)}{\om^\star(z)}<\infty$ for each analytic self-map $\vp$.
Therefore, each $\vp$ induces a bounded composition operator on $A^p_\om$ when $\om\in\DD$, and
the class $\DD$ induces in a sense a natural setting for the study of composition operators in weighted Bergman spaces.

We will use the characterization of the $q$-Carleson measures for $A^p_\om$~\cite{PelRatMathAnn} to describe bounded and compact composition operators. This involves maximal functions and tent spaces, and therefore we define the non-tangential approach regions
    \begin{equation*}\label{eq:gammadeuintro}
    \Gamma(\z)=\left\{z\in \D:\,|\t-\arg
    z|<\frac12\left(1-\frac{|z|}{r}\right)\right\},\quad
    \z=re^{i\theta}\in\D\setminus\{0\},
    \end{equation*}
and $N(f)(\z)=\sup_{z\in\Gamma(\z)}|f(z)|$.

\begin{theorem}\label{Theorem:introduction-bounded-composition-operators}
Let $0<p,q<\infty$, $\omega\in\DD$ and $v$ be a radial weight, and let $\vp$ be an
analytic self-map of $\D$.
\begin{itemize}
\item[\rm(a)] If $p>q$, then the following assertions are equivalent:
\begin{enumerate}
\item[\rm(i)] $C_\vp:A^p_\om\to A^q_v$ is bounded;
\item[\rm(ii)] $C_\vp:A^p_\om\to A^q_v$ is compact;
\item[\rm(iii)] $N\left(\frac{N_{\vp,v^\star}}{\om^\star}\right)\in L^\frac{p}{p-q}_\om$.
\end{enumerate}
\item[\rm(b)] If $q\ge p$, then $C_\vp:A^p_\om\to A^q_v$ is bounded if and only if
    $$
    \limsup_{|z|\to1^-}\frac{N_{\vp,v^\star}(z)}{\om^\star(z)^\frac{q}{p}}<\infty.
    $$
\item[\rm(c)] If $q\ge p$, then $C_\vp:A^p_\om\to A^q_v$ is compact if and only if
    $$
    \lim_{|z|\to1^-}\frac{N_{\vp,v^\star}(z)}{\om^\star(z)^\frac{q}{p}}=0.
    $$
\end{itemize}
\end{theorem}

Probably the least trivial part in Theorem~\ref{Theorem:introduction-bounded-composition-operators} is to see that (iii) in (a) is satisfied if the operator is bounded. The proof of this part relies strongly on the characterizations of the $q$-Carleson measures for $A^p_\om$ and bounded differentiation operators from $A^p_\om$ to $L^q_\mu$ in terms of tent spaces~\cite{PelRatMathAnn}. The condition~(iii) in the classical case $C_\vp:A^p_\a\to A^q_\b$ gives a characterization of bounded (and compact) operators that differs from the one in the existing literature~\cite{SmithYang98}. One can certainly obtain an analogue of \cite[Theorem~1.1]{SmithYang98} for the subclass of $\DD$ consisting of
weights with the property
    \begin{equation}\label{pikkupillu}
    \om(r)\asymp\frac{\int_r^1\om(s)\,ds}{1-r},\quad 0\le r<1,
    \end{equation}
by using factorization~\cite[Theorem~3.1]{PelRat} and an atomic decomposition~\cite[Theorem~4.1]{LuecInd85} of $A^p_\om$-functions together with maximal theorems, but this approach does not work for the whole class~$\DD$. Since these calculations do not serve our purposes we do not discuss this approach here any further. For details and more, see \cite{PelSummer}.

Although weighted Bergman spaces induced by $\om\in\DD$ may lie closer to the Hardy space $H^p$ than any $A^p_\alpha$, an analogue of Theorem~\ref{Theorem:introduction-bounded-composition-operators}(a) does not remain true for Hardy spaces. In fact, if $q<p$ and $\vp(z)=z$, then $C_\vp: H^p\to H^q$ is bounded but not compact.

We will prove the asymptotic formula $\|C_\vp\|_{(A^p_\om\to A^p_\om)}^p\asymp \frac{1}{\widehat{\omega}(\vp(0))(1-|\vp(0)|)}$ for the operator norm when $\om$ belongs to a natural subclass of $\DD$. Moreover, we will show that the quantity in Theorem~\ref{Theorem:introduction-bounded-composition-operators}(b) is comparable to the essential norm of a bounded composition operator $C_\vp: A^p_\om\to A^q_v$, and we will study the relation between the compactness of $C_\vp$ on $A^p_\om$ and the existence of the angular derivative of $\vp$ on the boundary of $\D$. Finally, we will obtain the following description of the Schatten class composition operators.

\begin{theorem}\label{Thm:intro-SchattenMain}
Let $0<p<\infty$ and $\omega\in\DD$, and let $\vp$ be an
analytic self-map of $\D$. Then $C_\vp\in\SSS_p(A^2_\omega)$ if
and only if
    \begin{equation}\label{37}
    \int_\D\left(\frac{N_{\vp,\omega^\star}(z)}{\omega^\star(z)}\right)^\frac{p}{2}\frac{dA(z)}{(1-|z|)^2}<\infty.
    \end{equation}
\end{theorem}

This result is an extension of \cite[Theorem~3]{LuZhu92} which concerns the classical case $C_\vp:A^2_\a\to A^2_\a$. It is known that in this case there are several ways to characterize the Schatten class composition operators~\cite{Zhu}.

\section{Integrability of reproducing kernels of Dirichlet spaces}

To state our main results concerning the $L^p$-behavior of reproducing kernels, we introduce
a family of Hilbert spaces. For and $\om\in\DD$, the Hilbert space $\Dom$ consists of
$f\in\H(\D)$ such that
    $$
    \|f\|^2_{\Dom}=|f(0)|^2\om(\D)+  \int_\D |f'(z)|^2\,\om(z)\,dA(z)<\infty.
    $$
We will work with the inner product on $\Dom$ defined by
    \begin{equation}
    \begin{split}\label{Eq:InnerProduct}
    \langle f,g\rangle_{\Dom}
    &=f(0)\overline{g(0)}\om(\D)+\int_\D
    f'(z)\overline{g'(z)}\,\om(z)\,dA(z).
    \end{split}
    \end{equation}
Let $K^{\omega}_a\in
\Dom$ be the corresponding reproducing kernels, that is,
    \begin{equation}\label{eq:repro}
    f(a)=f(0)\overline{K^{\omega}_a(0)}\om(\D)
    +\int_\D f'(z)\overline{\frac{\partial K^{\omega}_a(z)}{\partial z}}\,\om(z)\,dA(z),\quad f\in \Dom.
    \end{equation}
Write
    $$
    M_p^p(r,f)=\frac1{2\pi}\int_0^{2\pi}|f(re^{i\t})|^p\,d\t
    $$
and $M_\infty(r,f)=\max_{|z|=r}|f(z)|$. With these preparations we can state the main results of this section.

\begin{theorem}\label{th:kernelstimate}
Let $0<p<\infty$, $\om\in\DD$ and $N\in\N\cup\{0\}$. Then the following assertions hold:
\begin{enumerate}
\item[\rm(i)]$\displaystyle M_p^p\left(r,\left(K^{\omega}_a\right)^{(N)}\right)\asymp
    \int_0^{|a|r}\frac{dt}{\widehat{\om}(t)^{p}(1-t)^{p(N-1)}},\quad r,|a|\to1^-.$
\item[\rm(ii)] If $v\in\DD$, then
    \begin{equation}\label{k1}
    \|\left(K^{\omega}_a\right)^{(N)}\|^p_{A^p_v}
    \asymp \int_0^{|a|}\frac{\widehat{v}(t)}{\widehat{\om}(t)^{p}(1-t)^{p(N-1)}}\,dt,\quad |a|\to1^-.
    \end{equation}
\end{enumerate}
\end{theorem}

It is clear by the proof that the asymptotic inequality $\lesssim$ in \eqref{k1} is valid for any radial weight $v$. The following consequence of Theorem~\ref{th:kernelstimate} is often more useful than the theorem itself.

\begin{corollary}\label{co:ernelstimaten}
Let $0<p<\infty$, $\om\in\DD$ and $N\in\N\cup\{0\}$. Then the following assertions hold:
\begin{enumerate}
\item[\rm(i)]
    $\displaystyle
    M_p^p\left(r,\left(K^{\omega}_a\right)^{(N)}\right)\asymp \frac{1}{\widehat{\om}(ar)^{p}(1-|a|r)^{p(N-1)-1}},\quad r, |a|\to1^-,
    $\\
if and only if
    \begin{equation}\label{Eq:hypothesis-kernelmeans}
    \int_0^{|a|}\frac{dt}{\widehat{\om}(t)^{p}(1-t)^{p(N-1)}}\lesssim\frac{1}{\widehat{\om}(a)^{p}(1-|a|)^{p(N-1)-1}},\quad |a|\to1^-.
    \end{equation}
\item[\rm(ii)] If $v\in\DD$, then
\begin{equation}\label{kn4}
    \|\left(K^{\omega}_a\right)^{(N)}\|^p_{A^p_v}\asymp\frac{\widehat{v}(a)}{\widehat{\om}(a)^{p}(1-|a|)^{p(N-1)-1}},\quad |a|\to1^-,
    \end{equation}
    if and only if
    \begin{equation}\label{Eq:hypothesis-kernel}
    \int_0^r\frac{\widehat{v}(t)}{\widehat{\om}(t)^{p}(1-t)^{p(N-1)}}\,dt\lesssim\frac{\widehat{v}(r)}{\widehat{\om}(r)^{p}(1-r)^{p(N-1)-1}},\quad r\to1^-.
    \end{equation}
\end{enumerate}
\end{corollary}

Auxiliary results will be needed to prove Theorem~\ref{th:kernelstimate} and deduce Corollary~\ref{co:ernelstimaten}. The following lemma contains several properties of weights in $\DD$ and will be repeatedly used throughout the paper. The proof is straightforward and therefore omitted.

\begin{letterlemma}\label{Lemma:replacement-Lemmas-Memoirs}
Let $\om$ be a radial weight. Then the following conditions are equivalent:
\begin{itemize}
\item[\rm(i)] $\om\in\DD$;
\item[\rm(ii)] There exist $C=C(\om)>0$ and $\b_0=\b_0(\om)>0$ such that
    \begin{equation*}
    \begin{split}
    \widehat{\om}(r)\le C\left(\frac{1-r}{1-t}\right)^{\b}\widehat{\om}(t),\quad 0\le r\le t<1,
    \end{split}
    \end{equation*}
for all $\b\ge\b_0$;
\item[\rm(iii)] There exist $C=C(\om)>0$ and $\gamma_0=\gamma_0(\om)>0$ such that
    \begin{equation*}
    \begin{split}
    \int_0^t\left(\frac{1-t}{1-s}\right)^\g\om(s)\,ds
    \le C\widehat{\om}(t),\quad 0\le t<1,
    \end{split}
    \end{equation*}
for all $\g\ge\g_0$;
\item[\rm(iv)] The asymptotic equality
    $$
    \int_0^1s^x\om(s)\,ds\asymp\widehat{\om}\left(1-\frac1x\right),\quad x\in[1,\infty),
    $$
is valid;
\item[\rm(iv)] $\om^\star(z)\asymp\widehat{\om}(z)(1-|z|)$, $|z|\to1^-$,
where
    $$
    \omega^\star(z)=\int_{|z|}^1\omega(s)\log\frac{s}{|z|}s\,ds,\quad z\in\D\setminus\{0\}.
    $$
\end{itemize}
\end{letterlemma}

If $\{e_n\}$ is an orthonormal basis of a Hilbert
space $H$, that is continuously embedded into $\H(\D)$, then its reproducing kernel is given by
    \begin{equation}\label{RKformula}
    K_z(\zeta)=\sum_n e_ n(\zeta)\,\overline{e_ n(z)}
    \end{equation}
for all $z$ and $\zeta$ in $\D$. We denote $\om_x=\int_0^1r^{2x+1}\om(r)\,dr$ for $x>-1$.

\begin{lemma}\label{HL}
Let $\om\in\DD$ and $n\in\N\cup\{0\}$.
\begin{enumerate}
\item[\rm(i)] If $0<p\le2$, then
    $$
   M_p^p\left(r,\left(K^{\omega}_a\right)^{(N)}\right)\gtrsim
     \int_0^{|a|r}\frac{dt}{\widehat{\om}(t)^{p}(1-t)^{p(N-1)}},\quad r,|a|\to1^-.
    $$
  \item[\rm(ii)] If $2\le p<\infty$, then
    $$
    M_p^p\left(r,\left(K^{\omega}_a\right)^{(N)}\right)\lesssim
     \int_0^{|a|r}\frac{dt}{\widehat{\om}(t)^{p}(1-t)^{p(N-1)}},\quad r,|a|\to1^-.
    $$
\end{enumerate}
\end{lemma}

\begin{proof}
Write $\omega_\b(z)=(1-|z|)^\b\omega(z)$ for all
$\b\in\mathbb{R}$ and $z\in\D$. By using the standard orthonormal basis
$\{\left(\om(\D)\right)^{-\frac{1}{2}}\}\cup\{z^{n+1}/((n+1)\sqrt{2\om_{n}})\}$, $n\in \N\cup\{0\}$, of
$\Dom$ and \eqref{RKformula} we obtain
    \begin{equation}\label{kernelformula}
    K^{\omega}_a(z)=\om(\D)^{-1}+\sum_{n=0}^\infty\frac{z^{n+1}\overline{a}^{n+1}}{2(n+1)^2\om_n},
    \end{equation}
which implies
    \begin{equation*}
    \left(K^{\omega}_a\right)^{(N)}(z)=\sum_{n= N-1}^\infty
    \frac{n(n-1)\cdots(n-N+2)z^{n-N+1}\overline{a}^{n+1}}{2(n+1) \om_n},\quad N\in\N.
    \end{equation*}
Therefore the classical Hardy-Littlewood inequalities \cite[Theorem~6.2]{Duren1970} applied to $\left(K^{\omega}_a\right)^{(N)}(rz)$ show that it suffices to prove
    \begin{equation}\label{series}
    \sum_{n=N}^\infty\frac{r^{pn}}{(n+1)^{p(1-N)+2}\om_n^{p}}\asymp \int_0^{r}\frac{dt}{\widehat{\om}_{N-1}(t)^{p}},
    \quad r \to 1^-.
    \end{equation}
Assume, without loss of generality, that $r>1-\frac1{N+1}$. Choose now $N^\star\in\N$ such that $1-\frac{1}{N^\star}\le r<1-\frac{1}{N^\star+1}$. Then Lemma~\ref{Lemma:replacement-Lemmas-Memoirs} yields
    \begin{equation*}
    \begin{split}
    \sum_{n=N}^{N^\star}\frac{r^{pn}}{(n+1)^{p(1-N)+2}\om_n^{p}}
    &\asymp\sum_{n=N}^{N^\star}\frac{1}{(n+1)^{p(1-N)+2}\widehat{\om}(1-\frac{1}{2n+1})^{p}}
    \gtrsim
    \int_{N+1}^{\frac{1}{1-r}}\frac{ds}{s^{p(1-N)+2}\widehat{\om}(1-\frac{1}{s})^{p}}\\
    &=\int_{1-\frac{1}{N+1}}^{r}\frac{dt}{\widehat{\om}(t)^{p}(1-t)^{p(N-1)}}
    \asymp \int_{0}^{r}\frac{dt}{\widehat{\om}(t)^{p}(1-t)^{p(N-1)}},\quad r \to 1^-.
    \end{split}
    \end{equation*}
Since Lemma~\ref{Lemma:replacement-Lemmas-Memoirs} allows us to establish the same upper bound in a similar manner, we deduce
  \begin{equation*}
    \begin{split}\sum_{n=N}^{N^\star}\frac{r^{pn}}{(n+1)^{p(1-N)+2}\om_n^{p}}
   \asymp \int_{0}^{r}\frac{dt}{\widehat{\om}(t)^{p}(1-t)^{p(N-1)}},\quad r \to 1^-.
    \end{split}
    \end{equation*}
Lemma~\ref{Lemma:replacement-Lemmas-Memoirs} also implies
    \begin{equation}\label{bi}
    \frac{1}{\widehat{\om}(r)^{p}(1-r)^{p(N-1)-1}}
    \asymp\int_{\frac{4r-1}{3}}^r  \frac{dt}{\widehat{\om}(t)^{p}(1-t)^{p(N-1)}},\quad r \to 1^-,
    \end{equation}
and the existence of a constant $M=M(p,\om)>p+1$ such that $\frac{\widehat{\om}(r)^p}{(1-r)^M}$ is essentially increasing. Hence
    \begin{equation*}
    \begin{split}
    \sum_{n=N^\star}^\infty\frac{r^{pn}}{(n+1)^{p(1-N)+2}\om_n^{p}}
    &\lesssim \frac{1}{(N^\star+1)^M \widehat{\om}(1-\frac{1}{N^\star})^{p}}\sum_{n=N^\star}^\infty (n+1)^{p(N-1)-2+M}r^{np}\\
    &\asymp\frac{1}{(1-r)^{p(N-1)-1}\widehat{\om}(r)^{p}}
    \lesssim \int_{0}^{r}\frac{dt}{\widehat{\om}(t)^{p}(1-t)^{p(N-1)}},\quad r \to 1^-,
    \end{split}
    \end{equation*}
and the proof is complete.
\end{proof}

\medskip

With these preparations we are ready to establish Theorem~\ref{th:kernelstimate} by using a method employed in~\cite{PelRatproj}. Unfortunately, the Bergman kernel estimate \cite[Theorem~1]{PelRatproj} does not yield the result we wish to obtain, so we will follow its proof. In order to avoid unnecessary repetition of details, we will only indicate the main differences in the reasoning and omit the rest of the proof.

A weight $\om\in\DD$ is called regular if it satisfies \eqref{pikkupillu}. It is clear that $\om\in\R$ if and only if for each $s\in[0,1)$ there exists a constant $C=C(s,\omega)>1$ such that
    \begin{equation}\label{eq:r2}
    C^{-1}\om(t)\le \om(r)\le C\om(t),\quad 0\le r\le t\le
    r+s(1-r)<1,
    \end{equation}
and 
    \begin{equation*}
    \frac{\int_r^1\om(s)\,ds}{1-r}\lesssim \om(r),\quad0\le r<1.
    \end{equation*}

\medskip

\noindent\emph{Proof of }Theorem~\ref{th:kernelstimate}. We will first prove~(ii) for a continuous~$v\in\R$. Throughout the proof we may assume, without loss of generality, that $\int_0^1 \om(s)\,ds=1$. We recall the Littlewood-Paley formula
    \begin{equation*}
    \|f\|_{A^p_v}^p\asymp\int_{\D}|f^{(n)}(z)|^p(1-|z|)^{np}v(z)\,dA(z)+\sum_{j=0}^{n-1}|f^{(j)}(0)|^p,\quad f\in \H(\D),
    \end{equation*}
valid for regular continuous weights $v$, $0<p<\infty$ and $n\in\N$~\cite{PavP}. This implies
    \begin{equation*}
    \begin{split}
    \|(K^\om_a)^{(N)}\|_{A^p_v}^p&\asymp\int_{\D}|(K^\om_a)^{(N+2)}(z)|^p(1-|z|)^{2p}v(z)\,dA(z)
    +|(K^\om_a)^{(N)}(0)|^p+|(K^\om_a)^{(N+1)}(0)|^p,
    \end{split}
    \end{equation*}
where
    \begin{equation*}
    \left(K^{\omega}_a\right)^{(N+2)}(z)=\sum_{j= N+1}^\infty
    \frac{j(j-1)\cdots(j-N)z^{j-N-1}\overline{a}^{j+1}}{2(j+1)\om_j},\quad n\in\N,
    \end{equation*}
by \eqref{kernelformula}. Now, bearing in mind Lemma~\ref{HL}, and replacing $\left(B^\om_a\right)^{(N)}$
by $\left(K^{\omega}_a\right)^{(N+2)}$ in the proof of the case $v\in\R$ continuous in \cite[Theorem~1]{PelRatproj}, we get
    \begin{equation*}
    \begin{split}
    \|(K^\om_a)^{(N)}\|_{A^p_v}^p &\asymp  \|(K^\om_a)^{(N+2)}\|_{A^p_{v_{2p}}}^p
    \asymp\int_0^{|a|}\frac{\widehat{v_{2p}}(t)}{\widehat{\om}(t)^{p}(1-t)^{p(N+1)}}\,dt
    \asymp\int_0^{|a|}\frac{\widehat{v}(t)}{\widehat{\om}(t)^{p}(1-t)^{p(N-1)}}\,dt
    \end{split}
    \end{equation*}
as $|a|\to1^-$, and thus~(ii) for a continuous~$v\in\R$ is proved.

One readily obtains the assertion~(i) by following the lines of the corresponding part of the proof of \cite[Theorem~1]{PelRatproj} with Lemma~\ref{HL} and the case that we just proved in hand.

The case~(ii) for~$v\in\DD$ as well as the asymptotic inequality $\lesssim$ in \eqref{k1} for any radial
weight $v$ can also be proved by following the lines of the proof of \cite[Theorem~1]{PelRatproj}. We omit the details.
\hfill$\Box$

\medskip

\noindent\emph{Proof of }Corollary~\ref{co:ernelstimaten}. The equivalence between the asymptotic equality (i) and \eqref{Eq:hypothesis-kernelmeans} follows by Theorem~\ref{th:kernelstimate}(i) and \eqref{bi}. Moreover, \eqref{kn4} is equivalent to \eqref{Eq:hypothesis-kernel} by Theorem~\ref{th:kernelstimate}(ii) and the estimate
    \begin{equation*}
    \begin{split}
    \frac {\widehat{v}\left(a\right)}{\widehat{\om}(a)^{p}(1-|a|)^{(N-1)p-1}}&\asymp
    \int_{2|a|-1}^{|a|}\frac{\widehat{v}\left(s\right)}{\widehat{\om}(s)^p(1-s)^{(N-1)p}}\,ds
    \le \int_{0}^{|a|}\frac{\widehat{v}\left(s\right)}{\widehat{\om}(s)^p(1-s)^{(N-1)p}}\,ds,
    \end{split}
    \end{equation*}
that follows by Lemma~\ref{Lemma:replacement-Lemmas-Memoirs}.\hfill$\Box$

\section{Schatten class Toeplitz operators on Dirichlet spaces}\label{Section:Schatten-Toeplitz}

For $-\infty<\alpha<2$ and $\om\in\DD$, the Hilbert space $H_\alpha(\om^\star)$ consists of
$f\in\H(\D)$ such that
    $$
    \int_\D |f'(z)|^2\,\om^\star_{-\alpha}(z)\,dA(z)<\infty.
    $$
By \eqref{LP1}, we deduce the identity $H_0(\om^\star)=A^2_\omega$. For a complex Borel measure $\mu$ on $\D$,
the Toeplitz operator $T_\mu$ associated with $H_\alpha(\om^\star)$ is defined by
    \begin{equation*}
    T_\mu(f)(z)=\int_\D f(\z)K^{\alpha,\omega}(z,\z)\,d\mu(\z),\quad f\in
    H_\alpha(\om^\star),
    \end{equation*}
where $K^{\alpha,\omega}(z,\z)=K_\z^{\alpha,\omega}(z)$ are the corresponding reproducing kernels. Let now $\om$ be a radial weight and $\mu$ a $1$-Carleson measure for $A^2_\om$. By applying \eqref{RKformula} for the standard orthonormal bases and \eqref{LP1} for $f(z)=z^n$ we deduce
    \begin{equation*}
    T_\mu(f)(z)=4\mathcal{T}_\mu(f)+\left(\frac{1}{\om^\star(\D)}-\frac{4}{\om(\D)} \right)\|f\|_{L^1(\mu)},\quad f\in A^2_{\om}.
    \end{equation*}
Therefore $\mathcal{T}_\mu\in\SSS(A^2_\om)$ if and only if $T_\mu\in\SSS(H_0(\om^\star))$, and thus the following result contains Theorem~\ref{mainToeplitzbergman} as a special case. In the statement we use the notation
    $$
    \widehat{\mu}_{\a,r}(z)=\frac{\mu\left(\Delta(z,r)\right)}{\om_{-\a}^\star(z)},\quad z\in\D\setminus\{0\}.
    $$
\begin{theorem}\label{mainToeplitz}
Let $0<p<\infty$ and $-\infty<\a<1$ such that $p\a<1$. Let $\om\in\DD$ and $\mu$ be a positive Borel measure on~$\D$. Then the following assertions are equivalent:
\begin{enumerate}
\item[\rm (a)] $T_\mu\in\SSS_p(H_\alpha(\om^\star))$;
\item[\rm (b)]  $\sum_{R_j\in\Upsilon}
    \left(\frac{\mu(R_j)}{\omega^\star_{-\alpha}(z_j)}\right)^p<\infty$;
\item[\rm (c)]  $\widehat{\mu}_{\a,r}\in L^p\left(\frac{1}{(1-|z|)^2}\right)$ for some $0<r<1$.
\end{enumerate}
\end{theorem}

It is worth noticing that a straightforward calculation based on the pseudohyperbolic distance and the properties of regular weights show that condition (b) in Theorem~\ref{mainToeplitz} can be replaced by
    $\sum_{j}
    \left(\frac{\mu\left(\Delta(a_j,r)\right)}{\omega^\star_{-\alpha}(a_j)}\right)^p<\infty$,
where $\{a_j\}$ is a $\delta$-lattice. Recall that $A=\{a_k\}_{k=0}^\infty\subset\D$ is
uniformly discrete if it
is separated in the hyperbolic metric, it is an $\e$-net if $\D=\bigcup_{k=0}^\infty \Delta(a_k,\e)$, and
finally, it is a
$\delta$-lattice if it is a $5\delta$-net and uniformly
discrete with constant $\gamma=\delta/5$.

We begin with recalling the natural connection between the Toeplitz operator $T_\mu:H_\alpha(\om^\star)\to H_\alpha(\om^\star)$ and the identity operator $I_d$ acting from $H_\alpha(\om^\star)$ to the space $L^2(\mu)$.
Namely, the definition~\eqref{Eq:InnerProduct}, Fubini's
theorem and \eqref{eq:repro} give
    \begin{equation}\label{eq:st12}
    \begin{split}
    \langle T_\mu(g),f\rangle_{H_\alpha(\om^\star)}
    =\langle g,f\rangle_{L^2(\mu)}
    \end{split}
    \end{equation}
for all $g$ and $f$ in $H_\alpha(\om^\star)$, and therefore $T_\mu$ is bounded (resp.
compact) on $H_\alpha(\om^\star)$ if and only if
$I_d:\,H_\alpha(\om^\star)\to L^2(\mu)$ is bounded (resp.
compact).

The main effort on the way to Theorem~\ref{mainToeplitz} consists on proving the following result.

\begin{theorem}\label{th:tmuextended}
Let $0<p<\infty$ and $-\infty<\alpha<1$ such that $p\alpha<1$.
Let $\omega\in\DD$, and let $\mu$ be a complex Borel measure
on $\D$. If
    \begin{equation}\label{eq:slext}
    \sum_{R_j\in\Upsilon}
    \left(\frac{|\mu|(R_j)}{\omega^\star_{-\alpha}(z_j)}\right)^p<\infty,
    \end{equation}
then $T_\mu\in\SSS_p(H_\alpha(\om^\star))$, and there exists a
constant $C>0$ such that
\begin{equation*}
|T_\mu|_p^p\le C \sum_{R_j\in\Upsilon}
    \left(\frac{|\mu|(R_j)}{\omega^\star_{-\alpha}(z_j)}\right)^p.
    \end{equation*}
Conversely, if $\mu$ is a positive Borel measure on $\D$ and
$T_\mu\in\SSS_p(H_\alpha(\om^\star))$, then \eqref{eq:slext} is
satisfied.
\end{theorem}

The statement in Theorem~\ref{th:tmuextended} is valid if $p\ge1$ and $\om$ belongs to the subclass $\I\cup\R$ of~$\DD$ by~\cite[Theorem~6.11]{PelRat}.
The argument used there works for the class~$\DD$, and therefore here we only have to deal with the range $0<p<1$, where the proof is actually more involved. Note that in~\cite{PelRat} the space $H_\alpha(\om^\star)$ is defined by the condition
    \begin{equation*}
    \sum_{n=0}^\infty
    (n+1)^{\alpha+2}\om^\star_n|a_{n+1}|^2<\infty,\quad -\infty<\alpha<2,
    \end{equation*}
where $\sum_{n=0}^\infty a_n z^n$ is the Maclaurin series of $f\in\H(\D)$, and this condition is equivalent to our definition of $H_\alpha(\omega^\star)$.

Before proceeding to the auxiliary results needed, we show how Theorem~\ref{mainToeplitz} follows once Theorem~\ref{th:tmuextended} is proved.

\medskip

\noindent\emph{Proof of }Theorem~\ref{mainToeplitz}. The equivalence of (a) and (b) follows by Theorem~\ref{th:tmuextended}. Let now $r\in(0,1)$ be fixed. Divide each polar rectangle $R_j$ into $K^2$ disjoint subrectagles $R_j^k$, $k=1,\ldots,K^2$, of approximately equal size in the hyperbolic sense such that $R_j^k\subset\Delta(z,r)$ for all $z\in
R_j^k$, $k=1,\ldots,K^2$ and $R_j\in\Upsilon$. Then, by Lemma~\ref{le:sc1} below and \eqref{eq:r2},
    \begin{equation*}
    \begin{split}
    \sum_{R_j\in\Upsilon}\left(\frac{\mu(R_j)}{\omega^\star_{-\alpha}(z_j)}\right)^p
    &\asymp\sum_{R_j\in\Upsilon}\sum_{k=1}^{K^2}\left(\frac{\mu(R_j^k)}{\omega^\star_{-\alpha}(z_j)}\right)^p
    \asymp\sum_{R_j\in\Upsilon}\sum_{k=1}^{K^2}\frac{\int_{R_j^k}\left(\mu(R_j^k)\right)^p\frac{dA(z)}{(1-|z|)^2}}{(\omega^\star_{-\alpha}(z_j))^p}\\
    &\lesssim\sum_{R_j\in\Upsilon}\sum_{k=1}^{K^2}\int_{R_j^k}\left(\frac{\mu(\Delta(z,r))}{\omega^\star_{-\alpha}(z)}\right)^p\frac{dA(z)}{(1-|z|)^2}
    \asymp\int_\D\left(\widehat{\mu}_{\a,r}(z)\right)^p\frac{dA(z)}{(1-|z|)^2},
    \end{split}
    \end{equation*}
and hence (c) implies (b) is proved. The fact that (b) implies (c) can be proved by a reasoning similar to that above bearing in mind that each disc $\Delta(z,r)$ intersects at most $N=N(r)\in\N$ squares of the lattice $\Upsilon$. This finishes the proof of the theorem.\hfill$\Box$

\medskip

We will need several technical auxiliary results to prove Theorem~\ref{th:tmuextended}. The first two lemmas concern the regularity of associated weights. With the aid of Lemma~\ref{Lemma:replacement-Lemmas-Memoirs}, the first of them can be proved as~\cite[Lemma~1.7]{PelRat}.

\begin{lemma}\label{le:sc1}
If $0<\alpha<\infty$ and $\om\in\DD$, then
$\omega^\star_{\alpha-2}\in\R$ and
$\left(\omega^\star_{\alpha-2}\right)^\star(z)\asymp\omega^\star_{\alpha}(z)$
for all $|z|\ge\frac12$.
\end{lemma}

With Lemmas~\ref{Lemma:replacement-Lemmas-Memoirs} and~\ref{le:sc1} in hand, one readily obtains the following extension of \cite[Lemma~6.3]{PelRat} from its proof.

\begin{lemma}\label{le:sc2} If $0<\alpha<\infty$, $n\in\N$ and
$\om\in\DD$, then
    \begin{equation*}
    \int_0^1r^n\om_{\alpha-2}^\star(r)\,dr\asymp
    \frac{\om^\star\left(1-\frac{1}{n+1}\right)}{(n+1)^{\alpha-1}},
    \end{equation*}
and
    \begin{equation*}
    (n+1)^{2-\alpha}\om^\star_n\asymp\int_0^1r^{2n+1}\om_{\alpha-2}^\star(r)\,dr.
    \end{equation*}
\end{lemma}

The third auxiliary result needed is a consequence of the kernel estimates proved in the previous section.

\begin{corollary}\label{co:kernelstimate}
Let $0<p<\infty$, $\om\in\DD$, $-\infty<\a<1$ and $N\in\N$
such that $(1-|z|)^{Np-2}\left(\om_{-\alpha}^\star(z)\right)^{\frac{p}{2}}$ is a regular weight.
Then
    \begin{equation}\label{112}
    \int_{\D}\left|(K^{\alpha,\omega}_a)^{(N)}(z)\right|^p(1-|z|)^{Np-2}\left(\om_{-\alpha}^\star(z)\right)^{\frac{p}{2}}\,dA(z)
    \asymp\left(\om_{-\alpha}^\star(a) \right)^{-\frac{p}{2}},\quad |a|\ge\frac12.
    \end{equation}
\end{corollary}

\begin{proof}
Since $W(z)=(1-|z|)^{Np-2}\left(\om_{-\alpha}^\star(z)\right)^{\frac{p}{2}}$ is regular by the hypothesis, Lemma~\ref{Lemma:replacement-Lemmas-Memoirs} yields
    $$
    \widehat{W}(z)\asymp (1-|z|)^{Np-1}\left(\om_{-\alpha}^\star(z)\right)^{\frac{p}{2}}
    \asymp (1-|z|)^{p\left(N+\frac{1-\a}{2}\right)-1}\widehat{\om}(z)^{\frac{p}{2}}.
    $$
Another application Lemma~\ref{Lemma:replacement-Lemmas-Memoirs} gives
    \begin{equation*}\begin{split}
    \int_0^{|a|}
    \frac{\widehat{W}(s)}
    {(\widehat{\om_{-\alpha}^\star} (s))^p(1-s)^{p\left(N-1\right)}}\,ds
    & \asymp
    \int_0^{|a|}
    \frac{\widehat{W}(s)}
    {\widehat{\om}(s)^p(1-s)^{p\left(N+1-\a\right)}}\,ds
   \asymp\int_0^{|a|}\frac{1}{\widehat{\om}(s)^{\frac{p}{2}}(1-s)^{p\frac{1-\a}{2}+1}}\,ds
    \\ & \le
   \frac{1}{\widehat{\om}(a)^{\frac{p}{2}}}\int_0^{|a|}\frac{ds}{(1-s)^{p\frac{1-\a}{2}+1}}
   \asymp \frac{1}{\widehat{\om}(a)^{\frac{p}{2}}(1-|a|)^{p\frac{1-\a}{2}}}\\
   &\asymp
   \left(\om_{-\alpha}^\star(a) \right)^{-\frac{p}{2}},\quad |a|\ge\frac12.
   \end{split}
   \end{equation*}
Consequently, \eqref{112} follows from Corollary~\ref{co:ernelstimaten}.
\end{proof}

The next lemma shows that the derivatives of the reproducing kernel $K^{\alpha,\omega}_a$ are essentially constant in sufficiently small hyperbolic discs centered at the point $a$. The proof of this lemma is standard and therefore omitted, see \cite[Lemma~6.4]{PelRat} for a similar result.

\begin{lemma}\label{b6}
Let $\om\in\DD$, $-\infty<\alpha<1$ and $N\in\N\cup\{0\}$. Then there exists $r_0=r_0(\om,\alpha,N)\in(0,1)$
such that $ \left|\frac{\partial^N K^{\alpha,\omega}(a,z)}{\partial^N a }\right|\asymp \left|\left(\frac{\partial^N K^{\alpha,\omega}(a,z)}{\partial^N a}\right)\Big|_{z=a}\right|$ for all $a\in\D$ and
$z\in\Delta(a,r_0)$.
\end{lemma}

In the next lemma we define an auxiliary linear operator induced by a sequence $\{b_j\}\subset\D$ and a derivative of the reproducing kernel $K^{\a,\omega}_a$. It turns out that this operator is bounded from any separable Hilbert space to $H_\alpha(\om^\star)$ whenever $\{b_j\}$ is uniformly discrete, and even onto if $\{b_j\}$ is a
$\delta$-lattice.

\begin{lemma}\label{le:wo}
Let $\om\in\DD$, $-\infty<\a<1$ and $N\in\N$. For $a\in\D\setminus\{0\}$,
define
    $$
    h^{N,\om_{-\alpha}^\star}_{a}(z)=(1-|a|)^{N}\left(\om_{-\alpha}^\star(a)\right)^{\frac{1}{2}}\frac{\partial^N K^{\a,\omega}(z,a)}{\partial^N \overline{a}},\quad z\in\D,
    $$
and set $h^{N,\om_{-\alpha}^\star}_0\equiv 0$.
Let $\{b_j\}_{j=0}^\infty$ be a uniformly discrete sequence
ordered by increasing moduli, and let $\{e_ j\}_{j=0}^\infty$ be an
orthonormal basis of a Hilbert space $H$. Let $J$ be the linear
operator such that
$J(e_0)=\frac{1}{\left(\om_{-\alpha}^\star(\D)\right)^{1/2}}$, $J(e_j)=\frac{z^j}{j^2\left(\om^\star_{-\alpha}\right)_{j-1}^{1/2}}$ for all $j=1,\dots, N-1$, and
$J(e_j)=h^{N,\om_{-\alpha}^\star}_{b_j}$ if $j\ge N$. Then $J:H\to
H_\alpha(\om^\star)$ is bounded. Moreover, if $\{b_j\}$ is a
$\delta$-lattice for some $\delta\in (0,1)$, then $J$ is onto.
\end{lemma}

\begin{proof}
We first observe that $J:H\to
H_\alpha(\om^\star)$ is bounded if
    \begin{equation*}
    \left \| J \left(\sum_{j} c_ j e_ j \right) \right \|_{H_\alpha(\om^\star)}\lesssim \left
    (\sum_ j |c_ j|^2\right )^{1/2}
    \end{equation*}
for all sequences $\{c_j\}$ in $\ell^2$. Next, $N$ differentiations of
the reproducing formula \eqref{eq:repro} give
    \begin{equation}\label{eq:reproN}
    f^{(N)}(a)=
    \int_\D f'(z)\frac{\partial^{N+1} K^{\alpha,\omega}(a,z)}{\partial^N a\partial \overline{z}}\,\om^\star_{-\alpha}(z)\,dA(z),\quad a\in\D,\quad f\in H_\alpha(\om^\star).
    \end{equation}
By \eqref{kernelformula},
$h^{N,\om_{-\alpha}^\star}_{a}$ vanishes at the origin
for all $a\in\D$. Therefore, if $f(z)=\sum_{n=0}^\infty a_n z^n\in
H_\a(\om^\star)$, then \eqref{eq:reproN} and the Cauchy-Schwarz
inequality yield
   \begin{equation*}
   \begin{split}
    \left|\left \langle J \left(\sum_{j=0}^\infty c_ j e_ j \right), f\right
    \rangle_{H_\a(\om^\star)} \right|
 &   =\left|\left\langle \frac{c_0}{\left(\om_{-\alpha}^\star(\D)\right)^{1/2}}+\sum_{j=1}^{N-1}c_j\frac{z^j}{j^2\left(\om^\star_{-\alpha}\right)_{j-1}^{1/2}}
    +\sum_{j =N}^\infty c_ j h^{N,\om_{-\alpha}^\star}_{b_j}
    ,f \right \rangle_{H_\a(\om^\star)} \right |\\
    &=\Bigg|c_0\overline{a}_0\left(\om_{-\alpha}^\star(\D)\right)^{1/2}+2\left(\sum_{j=1}^{N-1}c_j\overline{a}_j(\om_{-\a}^\star)^{1/2}_{j-1}\right)\\ &\quad +\sum_{j =N}^\infty c_ j(1-|b_j|)^{N}\left(\om_{-\alpha}^\star(b_j)\right)^{\frac{1}{2}}\overline{f^{(N)}(b_j)}
    \Bigg|\\
    &\lesssim\|\{c_j\}\|_{\ell^2}\left(\sum_{j=0}^{N-1}|a_j|^2+\sum_{j=N}^\infty
    (1-|b_j|)^{2N}\om_{-\alpha}^\star(b_j)|f^{(N)}(b_j)|^2\right)^{1/2}.
    \end{split}
    \end{equation*}
Now, since $|f^{(N)}|^2$ is subharmonic, \eqref{eq:r2} for the
regular weight $\om_{-\alpha+2N-2}^\star$ and the assumption that
$\{b_j\}$ is uniformly discrete, give
    \begin{equation*}
    \begin{split}
    \sum_{j=N}^\infty (1-|b_j|)^{2N}\om_{-\alpha}^\star(b_j)|f^{(N)}(b_j)|^2 &\lesssim
    \int_{\D}|f^{(N)}(z)|^2\om_{-\alpha+2N-2}^\star(z)\,dA(z).
    \end{split}
    \end{equation*}
Consequently, $N-1$ applications of \eqref{LP1} show that
$J:H\to H_\a(\om^\star)$ is bounded.

Let now $\{b_j\}$ be a $\delta$-lattice for some $\delta\in
(0,1)$. To see that in this case the operator $J:H\to H_\a(\om^\star)$ is onto,
note first that the adjoint of $J$ is
    $$
    J^\star(f)=a_0\left(\om_{-\alpha}^\star(\D)\right)^{1/2}e_0
    +2\sum_{j=1}^{N-1}a_j(\om_{-\a}^\star)^{1/2}_{j-1}e_j
    +\sum_{j=N}^\infty (1-|b_j|)^{N}\left(\om_{-\alpha}^\star(b_j)\right)^{\frac{1}{2}}f^{(N)}(b_j)e_j.
    $$
Since the weight $\om_{-\alpha+2N-2}^\star$ is regular, it satisfies the
$C_p$ property defined at \cite[p.~321]{LuecInd85} by \eqref{eq:r2}. Therefore
\cite[Theorem~3.14]{LuecInd85} and $N-1$ applications of \eqref{LP1} imply that $J^\star$ is bounded below, and hence injective. In particular, $\textrm{Ker}\,(J^\star)=\{0\}$, and thus $J$ is onto.
\end{proof}

We are now ready to prove the main result of this section. The reasoning employed in the proof is based on ideas from~\cite{Lu87}.

\medskip

\noindent\emph{Proof of }Theorem~\ref{th:tmuextended}. A careful inspection of the proof of \cite[Theorem~6.11]{PelRat} with Lemmas~\ref{Lemma:replacement-Lemmas-Memoirs},~\ref{le:sc1},~\ref{le:sc2} and Corollary~\ref{co:kernelstimate} in hand shows that it can be carried over to the class $\DD$. Therefore it suffices to consider the case $0<p<1$.

Each regular weight is comparable to a continuous weight by the definition, and hence \cite[p. 10 (ii)]{PelRat} shows that we can fix $N=N(p,\alpha,\om)$ large enough such that $(1-|z|)^{Np-2}\left(\om_{-\alpha}^\star(z)\right)^{\frac{p}{2}}$ is a regular weight. This simple observation is one of the key steps in the proof.

Let $\mu$ be a complex Borel measure on $\D$ such that \eqref{eq:slext} is
satisfied. By the extension of \cite[Theorem~6.11]{PelRat} to $\DD$, $T_\mu\in
\SSS_1(H_\alpha(\om^\star))$ and, in particular, $T_\mu$ is
compact on $H_\a(\om^\star)$. We will show that
$T_\mu\in\SSS_p(H_\alpha(\om^\star))$. To do this, Corollary~\ref{co:kernelstimate} will be used.

Let $\{b_j\}_{j=0}^\infty$ be a
$\delta$-lattice. Then, for a fixed basis $\{e_j\}_{j=0}^\infty$
of $H_\a(\om^\star)$, the operator $J:H_\a(\om^\star)\to
H_\a(\om^\star)$, defined on Lemma~\ref{le:wo}, is bounded and
onto. Therefore, by \cite[Proposition~1.30]{Zhu}, $T_\mu\in\SSS_p(H_\alpha(\om^\star))$
if and only if $J^\star T_\mu J \in\SSS_p(H_\alpha(\om^\star))$. This together with
\cite[Lemma~5]{Lu87} (see also
\cite[Proposition~1.29]{Zhu}) shows that it suffices to prove
    \begin{equation*}
    \sum_{j=0}^\infty\sum_{k=0}^\infty\left|\langle T_\mu(h_j),h_k \rangle \right|^p\lesssim\sum_{R_j\in\Upsilon}
    \left(\frac{|\mu|(R_j)}{\omega^\star_{-\alpha}(z_j)}\right)^p,
    \end{equation*}
where $J(e_j)=h_j$.
To see this, we first observe that \eqref{eq:st12} yields
    \begin{equation*}
    \begin{split}
    \left|\langle T_\mu(h_j),h_k \rangle \right|
    &\le \int_{\D} |h_j(z)||h_k(z)|\,d|\mu|(z)
    =\sum_{R_n\in\Upsilon} \int_{R_n} |h_j(z)||h_k(z)|\,d|\mu|(z)\\
    &\le\sum_{R_n\in\Upsilon} |h_j(\wtz_{j,n})||h_k(\wtz_{k,n})||\mu|(R_n),
    \end{split}
    \end{equation*}
where $|h_j(\wtz_{j,n})|=\max_{z\in \overline{R}_n} |h_j(z)|$ for
each $j\in\N\cup\{0\}$. Since $0<p<1$ by the assumption, we deduce
    \begin{equation*}
    \begin{split}
    \sum_{j=0}^\infty\sum_{k=0}^\infty\left|\langle T_\mu(h_j),h_k \rangle \right|^p
    &\le  \sum_{j=0}^\infty\sum_{k=0}^\infty \sum_{R_n\in\Upsilon} |h_j(\wtz_{j,n})|^p|h_k(\wtz_{k,n})|^p|\mu|(R_n)^p\\
    & = \sum_{R_n\in\Upsilon}|\mu|(R_n)^p\left(\sum_{j=0}^\infty\sum_{k=0}^\infty
    |h_j(\wtz_{j,n})|^p|h_k(\wtz_{k,n})|^p\right).
    \end{split}
    \end{equation*}
Consequently, the proof will be finished once we prove
    \begin{equation}\label{eq:tmu2}
    \sum_{j=0}^\infty |h_j(\wtz_{j,n})|^p\le C \left(\om^\star_{-\a}(z_n)\right)^{-\frac{p}{2}},
    \end{equation}
where the constant $C>0$ is independent of $n$. Clearly,
    \begin{equation*}
    \begin{split}
    \sum_{j=0}^\infty |h_j(\wtz_{j,n})|^p
    &\lesssim N+ \sum_{j=N}^\infty (1-|b_j|)^{Np}\left(\om_{-\alpha}^\star(b_j)\right)^{\frac{p}{2}}
    \left|\frac{\partial^N K^{\a,\omega}(\wtz_{j,n},b_j)}{\partial^N \overline{b}_j}\right|^p\\
    &=N+ \sum_{j=N}^\infty (1-|b_j|)^{Np}\left(\om_{-\alpha}^\star(b_j)\right)^{\frac{p}{2}}
    \left|\frac{\partial^N K^{\a,\omega}(b_j,\wtz_{j,n})}{\partial^N
    b_j}\right|^p.
    \end{split}
    \end{equation*}
Since $\{\wtz_{j,n}\}\subset \overline{R_n}$, by applying the subharmonicity to the second variable of
$\frac{\partial^N K^{\a,\omega}(z,w)}{\partial^N z}$, and then to
the first one, and using \eqref{eq:r2} for the regular weight
$\om_{-\alpha}^\star$, we get
\begin{equation*}
\begin{split}
    &\sum_{j=N}^\infty (1-|b_j|)^{Np}\left(\om_{-\alpha}^\star(b_j)\right)^{\frac{p}{2}}
    \left|\frac{\partial^N K^{\a,\omega}(b_j,\wtz_{j,n})}{\partial^N
    b_j}\right|^p\\
    &\lesssim\frac{1}{(1-|z_n|)^2}\int_{\Delta(z_n,r)}\left(\sum_{j=N}^\infty (1-|b_j|)^{Np}\left(\om_{-\alpha}^\star(b_j)\right)^{\frac{p}{2}}
    \left|\frac{\partial^N K^{\a,\omega}(b_j,\z)}{\partial^N
    b_j}\right|^p\right)\,dA(\z)\\
    &\lesssim\frac{1}{(1-|z_n|)^2}\int_{\Delta(z_n,r)}\left(\sum_{j=N}^\infty (1-|b_j|)^{Np-2}
    \left(\om_{-\alpha}^\star(b_j)\right)^{\frac{p}{2}}\int_{\Delta(b_j,r)}
    \left|\frac{\partial^N K^{\a,\omega}(z,\z)}{\partial^N
    z}\right|^p\,dA(z)\right)\,dA(\z)\\
    &\lesssim\frac{1}{(1-|z_n|)^2}\int_{\Delta(z_n,r)}\left(\sum_{j=N}^\infty \int_{\Delta(b_j,r)}
    \left|\frac{\partial^N K^{\a,\omega}(z,\z)}{\partial^N
    z}\right|^p (1-|z|)^{Np-2}
    \left(\om_{-\alpha}^\star(z)\right)^{\frac{p}{2}}\,dA(z)\right)\,dA(\z)
    \end{split}
\end{equation*}
for a suitably chosen $r\in(0,1)$. Finally, by using that $\{b_j\}$ is uniformly discrete and applying
Corollary~\ref{co:kernelstimate}, we deduce
\begin{equation*}
\begin{split}
&\sum_{j=N}^\infty (1-|b_j|)^{Np}\left(\om_{-\alpha}^\star(b_j)\right)^{\frac{p}{2}}
    \left|\frac{\partial^N K^{\a,\omega}(b_j,\wtz_{j,n})}{\partial^N
    b_j}\right|^p\\
    &\lesssim\frac{1}{(1-|z_n|)^2}\int_{\Delta(z_n,r)}\left(\int_\D\left|\frac{\partial^N K^{\a,\omega}(z,\z)}{\partial^N
    z}\right|^p(1-|z|)^{Np-2}
    \left(\om_{-\alpha}^\star(z)\right)^{\frac{p}{2}}\,dA(z)\right)\,dA(\z)\\
    &\lesssim\frac{1}{(1-|z_n|)^2}\int_{\Delta(z_n,r)}\frac{dA(\z)}{(\om_{-\a}^\star(\z))^\frac{p}{2}}\asymp\frac{1}{(\om^\star_{-\a}(z_n))^\frac{p}{2}}.
\end{split}
\end{equation*}
The inequality \eqref{eq:tmu2} follows, and thus the first
assertion is proved.

Before proving the second part of the assertion, we introduce the necessary notation.
Each $a\in\D\setminus\{0\}$ induces the polar rectangle
    $$
    R(a)=\left\{z\in\D:\,
    |\arg a{\overline z}|<\frac{1-|a|}{2},\,\,\,\, |a|\le |z|<\frac{1+|a|}{2}\right\}.
    $$
For convenience, we also denote by $R_{-1}(a)$ the rectangle induced by the point $(2|a|-1)e^{i(\arg a-(1-|a|))}$, and if $|a|<\frac{1}{2}$, we set $R_{-1}(a)=D(0,1/2)$. Further, we denote
    $$
    Q(a)=R(a)\cup R_{-1}(a),\quad a\in\D\setminus\{0\}.
    $$

Let now $\mu$ be a positive Borel measure on $\D$, and assume that
$T_\mu\in \SSS_p(H_\alpha(\om^\star))$, where $0<p<1$. For each
$\e>0$ there exists $M=M(\e)\in\N$ such that $\{z_k\}$ (the sequence which induces the partition $\{R(I):\, I\in\Upsilon\}$) can be divided
into $M$ subsequences $\{z_k^{(j)}\}$, $j=1,\ldots,M$, such that
$\rho(z_n^{(j)},z_k^{(j)})>1-\e$ for all $n\ne k$ and
$j=1,\ldots,M$. Therefore, if $R_k^{(j)}\in\Upsilon$ denotes the
element containing $z_k^{(j)}$, then
$\rho(Q_n^{(j)},Q_k^{(j)})>1-\d(\e)$, for all $n\ne k$ and
$j=1,\ldots,M$, and $\d=\d(\e)\to0$, as $\e\to0$. The choice of
$\e$ will be made later.

Next, we need to do a new partition in order to ensure that the
kernels involved in the proof are essentially constant on certain subrectangles (small regions in the hyperbolic sense).
To do this, divide each $R_k^{(j)}\in\Upsilon$, $k\in\N$,
$j=1,\ldots,M$, into $P^2$ disjoint rectangles $\{R(z_k^{(j,l)})\}_{l=1}^{P^2}$
(of approximately equal size in
the hyperbolic sense) where
$P\in\N$ is sufficiently large so that we may use Lemma~\ref{b6} to deduce
    \begin{equation}\label{kercomp}
    \begin{split}
    \left|\frac{\partial^NK^{\a,\omega}(z_k^{(j,l)},z)}{\partial^N z_k^{(j,l)}}\right|
    &\asymp\left|\left(\frac{\partial^NK^{\a,\omega}(z_k^{(j,l)},z)}{\partial^N z_k^{(j,l)}}\right)\Bigg|_{z=z_k^{(j,l)}}\right|\\
    &\asymp\frac{1}{(1-|z_k^{(j,l)}|)^N\om^\star_{-\a}(z_k^{(j,l)})},\quad
    z\in R(z_k^{(j,l)}),
    \end{split}
    \end{equation}
where the constants of comparison do not depend on $k$, $j$ and $l$.
Moreover, we obviously have
$$
\rho(Q(z_n^{(j,l)}),Q(z_k^{(j,l)}))>1-\d,\quad n\ne k,\quad j=1,\ldots,M,\quad l=1,\ldots,P^2.
$$

Let $\mu_{j,l}=\left(\sum_k\chi_{R(z_k^{(j,l)})}\right)\mu$. We note that
$T_\mu$ and $T_{\mu_{j,l}}$ are both diagonizable as positive
operators and $|T_{\mu_{j,l}}|_p\le|T_\mu|_p$, see
\cite[p.~359]{Lu87} for details. 
Fix now indexes $j$ and $l$, and write $\nu=\mu_{j,l}$ and $b_k=z_{k}^{(j,l)}$ for
short. Write also $h_k=h^{N,\om_{-\alpha}^\star}_{b_k}$, and let
$J$ be the operator such that $J(e_j)=0$ if $j<N$ and $J(e_j)=h_j$ if $j\ge N$.
Then, by Lemma~\ref{le:wo}, $J^\star T_\nu
J\in \SSS_p(H_\alpha(\om^\star))$, whenever
$T_\nu\in\SSS_p(H_\alpha(\om^\star))$. Further, $J^\star T_\nu
J=D+E$, where $D$ is the diagonal operator
    $$
    D(f)=\sum_k\langle T_\nu(h_k),h_k\rangle_{H_\alpha(\om^\star)}\langle
    f,e_k\rangle_{H_\alpha(\om^\star)}
    e_k,
    $$
and $E$ is the remainder
    $$
    E(f)=\sum_{n}\sum_{k\ne n}\langle T_\nu(h_k),h_n\rangle_{H_\alpha(\om^\star)}\langle
    f,e_k\rangle_{H_\alpha(\om^\star)}
    e_n.
    $$
Now, bearing in mind \eqref{kercomp}, we deduce
    \begin{equation*}
    \begin{split}
    |D|_p^p&=\sum_k\langle
    T_\nu(h_k),h_k\rangle^p_{H_\alpha(\om^\star)}=\sum_k\left(\int_\D|h_k(z)|^2\,d\nu(z)\right)^p
    \ge\sum_k\left(\int_{R(b_k)}|h_k(z)|^2\,d\nu(z)\right)^p
    \\ &  \ge c_1\sum_k\left((1-|b_k|)^{2N}\om_{-\a}^\star(b_k)\frac{\nu (R(b_k))}{(1-|b_k|)^{2N}(\om_{-\a}^\star(b_k))^2}\right)^p
    = c_1\sum_k\left(\frac{\nu (R(b_k))}{\om_{-\a}^\star(b_k)}\right)^p
    \end{split}
    \end{equation*}
for some constant $c_1>0$ depending only on $\a$, $\om^\star$,
$N$ and $p$, and
where $R(b_k)$ is the square  induced by  $b_k$.

To deal with $E$, we may argue as in the
first part of the proof to obtain
    \begin{equation*}
    \begin{split}
    |E|_p^p&\le\sum_n\sum_{k\ne n}|\langle T_\nu(h_k),h_n\rangle_{H_\alpha(\om^\star)}|^p
    \le \sum_n\sum_{k\ne n}\left(\sum_j\int_{R(b_j)}|h_k(z)||h_n(z)|\,d\nu(z)\right)^p\\
    &=\sum_n\sum_{k\ne n}\Bigg(\sum_j(1-|b_k|)^N(\om_{-\a}^\star(b_k))^\frac12
    (1-|b_n|)^N(\om_{-\a}^\star(b_n))^\frac12\\
    &\quad\cdot\int_{R(b_j)}
    \left|\frac{\partial^N K^{\a,\omega}(z,b_k)}{\partial^N\overline{b}_k}\right|
    \left|\frac{\partial^N
    K^{\a,\omega}(z,b_n)}{\partial^N\overline{b}_n}\right|\,d\nu(z)\Bigg)^p\\
    &\le\sum_j\nu(R(b_j))^p\Bigg(\sum_n\sum_{k\ne n}(1-|b_k|)^{pN}(\om_{-\a}^\star(b_k))^\frac{p}2
    (1-|b_n|)^{pN}(\om_{-\a}^\star(b_n))^\frac{p}2\\
    &\quad\cdot\left|\frac{\partial^N
    K^{\a,\omega}(\widetilde{z}_{k,j},b_k)}{\partial^N\overline{b}_k}\right|^p
    \left|\frac{\partial^N
    K^{\a,\omega}(\widetilde{z}_{n,j},b_n)}{\partial^N\overline{b}_n}\right|^p\Bigg),
    \end{split}
    \end{equation*}
where $|\frac{\partial^N
    K^{\a,\omega}(\widetilde{z}_{k,j},b_k)}{\partial^N\overline{b}_k}|=\max_{z\in \overline{R(b_j)}}|\frac{\partial^N
    K^{\a,\omega}(z,b_k)}{\partial^N\overline{b}_k}|$
     for each $j\in \N\cup\{0\}$.

Next, bearing in mind that $\om^\star_{-\a}$ is regular, by using subharmonicity, we deduce
\begin{equation*}\begin{split}
&\sum_{k}(1-|b_k|)^{pN}(\om_{-\a}^\star(b_k))^\frac{p}2
    \left|\frac{\partial^N
    K^{\a,\omega}(\widetilde{z}_{k,j},b_k)}{\partial^N\overline{b}_k}\right|^p
 \\ &\lesssim \sum_{k} \int_{\Delta(b_k,r)}(1-|u|)^{Np-2}(\om_{-\a}^\star(u))^\frac{p}2
    \left|\frac{\partial^N
    K^{\a,\omega}(\widetilde{z}_{k,j},u)}{\partial^N\overline u}\right|^p\,dA(u).
\end{split}\end{equation*}
Choose now constants $0<s_1<s$ such that $\Delta(z,s_1)\subset \Delta(b_j,s)$ for any $z\in \overline{R(b_j)}$, $j\in\N$.
Then, by using subharmonicity again, we get
\begin{equation*}\begin{split}
&\sum_{k}(1-|b_k|)^{pN}(\om_{-\a}^\star(b_k))^\frac{p}2
    \left|\frac{\partial^N
    K^{\a,\omega}(\widetilde{z}_{k,j},b_k)}{\partial^N\overline{b}_k}\right|^p
     \\ &\lesssim
     \sum_{k} \int_{\Delta(b_k,r)}   (\om_{2N-\frac{4}p-\a}^\star(u))^\frac{p}2
     \left(\frac{1}{(1-|b_j|)^2}\int_{\Delta(b_j,s)} \left|\frac{\partial^N
    K^{\a,\omega}(z,u)}{\partial^N\overline u}\right|^p\,dA(z)\right)
   \,dA(u).
  \end{split}\end{equation*}
Next, we choose $r\in (0,1)$ such that $\Delta(z,r)\subset Q(z)$, $z\in\D\setminus\{0\}$. Since
the sets $\{Q(b_k)\}$ are disjoint, we get
\begin{equation*}\begin{split}
 &\sum_{k}(1-|b_k|)^{pN}(\om_{-\a}^\star(b_k))^\frac{p}2
    \left|\frac{\partial^N
    K^{\a,\omega}(\widetilde{z}_{k,j},b_k)}{\partial^N\overline{b}_k}\right|^p
 \\ &\lesssim
     \ \int_{\cup_{k}Q(b_k)}   (\om_{2N-\frac{4}p-\a}^\star(u))^\frac{p}2
     \left(\frac{1}{(1-|b_j|)^2}\int_{\Delta(b_j,s)} \left|\frac{\partial^N
    K^{\a,\omega}(z,u)}{\partial^N\overline u}\right|^p\,dA(z)\right)
   \,dA(u),
 \end{split}\end{equation*}
 where the constant in the asymptotic inequality depends on $r$, $s_1$ and $s$.
So, by Fubini's theorem the double sum $\sum_n\sum_{k\ne n}$ is dominated by a
positive constant times
    \begin{equation*}
    \begin{split}
    & \iint_{\cup_n\cup_{n\ne k}Q(b_k)\times
Q(b_n)}(\om_{-\a+2N-\frac{4}p}^\star(u))^\frac{p}2(\om_{-\a+2N-\frac{4}p}^\star(v))^\frac{p}2
\\ &\cdot  \left(\frac{1}{(1-|b_j|)^4}\iint_{\Delta(b_j,\delta)\times \Delta(b_j,\delta)} \left|\frac{\partial^N
    K^{\a,\omega}(z,u)}{\partial^N\overline u}\right|^p \left|\frac{\partial^N
    K^{\a,\omega}(\z,v)}{\partial^N\overline v}\right|^p\,dA(z)\,dA(\z)\right)
\,dA(u)\,dA(v)
\\ &\le \iint_{G}(\om_{-\a+2N-\frac{4}p}^\star(u))^\frac{p}2(\om_{-\a+2N-\frac{4}p}^\star(v))^\frac{p}2
\\ &\cdot  \left(\frac{1}{(1-|b_j|)^4}\iint_{\Delta(b_j,\delta)\times \Delta(b_j,\delta)} \left|\frac{\partial^N
    K^{\a,\omega}(z,u)}{\partial^N\overline u}\right|^p \left|\frac{\partial^N
    K^{\a,\omega}(\z,v)}{\partial^N\overline v}\right|^p\,dA(z)\,dA(\z)\right)
\,dA(u)\,dA(v)
\\ & \le \frac{1}{(1-|b_j|)^4}\iint_{\Delta(b_j,\delta)\times \Delta(b_j,\delta)}
\\ & \left[\iint_{G}(\om_{-\a+2N-\frac{4}p}^\star(u))^\frac{p}2(\om_{-\a+2N-\frac{4}p}^\star(v))^\frac{p}2
  \left|\frac{\partial^N
    K^{\a,\omega}(z,u)}{\partial^N\overline u}\right|^p
\left|\frac{\partial^N
    K^{\a,\omega}(\z,v)}{\partial^N\overline v}\right|^p\,dA(u)\,dA(v)\right]
\,dA(z)\,dA(\z),
    \end{split}
    \end{equation*}
where $G=\{(u,v):\rho(u,v)>1-\d\}\supset\cup_n\cup_{k\ne n}Q(b_k)\times
Q(b_n).$
Since we have chosen $N$ sufficiently large such that $(\om_{-\a+2N-\frac{4}p}^\star)^\frac{p}2$ is a regular weight,
Corollary~\ref{co:kernelstimate} shows that the double inner integral
is uniformly bounded by a positive constant times
    $$
    (\om_{-\a}^\star(z))^{-\frac{p}2}
    (\om_{-\a}^\star(\z)^{-\frac{p}2}
    \asymp(\om_{-\a}^\star(b_j))^{-p},
    $$
and hence
    $$
    \sum_n\sum_{k\ne n}(\cdot)=o((\om_{-\a}^\star(b_{j}))^{-p}),\quad
    \d\to0^+.
    $$
It follows that for each $\eta>0$ there exists $M=M(\eta)\in\N$ such that
    $$
    |E|_p^p\le\eta\sum_k\left(\frac{\mu (R(z^{j,l}_k))}{\om_{-\a}^\star(z^{j,l}_k)}\right)^p,\quad
    j=1,\ldots,M.
    $$
By choosing now $M$ big enough, we get $\eta$ small enough for which
    $$
   \sum_k\left(\frac{\mu (R(z^{j,l}_k))}{\om_{-\a}^\star(z^{j,l}_k)}\right)^p\lesssim| J^\star T_{\mu_{j,l}}J|_p^p
   \lesssim | T_{\mu_{j,l}}|_p^p \le|T_\mu|_p^p
    $$
for each $j=1,\ldots,M$ and $l=1,\dots, P^2$. If now
$\mu_j=\left(\sum_k\chi_{R_k^{(j)}}\right)\mu$, then
    \begin{equation*}
    \begin{split}
    \sum_k\left(\frac{\mu_j(R_k)}{\om_{-\a}^\star(z_k)}\right)^p
    &=\sum_k\left(\frac{\sum_{l=1}^{P^2}\mu(R_k^{(j,l)})}{\om_{-\a}^\star(z_k)}\right)^p
    \le\sum_k\sum_{l=1}^{P^2}\left(\frac{\mu(R_k^{(j,l)})}{\om_{-\a}^\star(z_k)}\right)^p\\
    &\asymp\sum_k\sum_{l=1}^{P^2}\left(\frac{\mu(R_k^{(j,l)})}{\om_{-\a}^\star(z_k^{(j,l)})}\right)^p
    \lesssim P^2 |T_\mu|_p^p,\quad j=1,\ldots,M.
    \end{split}
    \end{equation*}
This being true for each $j$, we get the assertion for compactly
supported $\mu$. If $\mu$ has not a compact support, then we may apply this to
$\mu_r=\chi_{D(0,r)}\mu$, and then use standard arguments 
to deduce $|T_{\mu_{_r}}|_p\le|T_\mu|_p$, and finally let $r\to1^-$ to complete the proof.\hfill$\Box$

\section{Bounded and compact composition operators}

For $0\le r<1$, let $n(r,z)=n_\vp(r,z)$ denote the number of preimages of $z$ under $\vp$ in $D(0,r)$,
and define the partial counting functions induced by a radial weight $v$ for $\vp$ by
    $$
    N_{\vp,v^\star}(r,z)=\sum_{\z\in\vp^{-1}(z),|\z|\le r}v^\star\left(\frac{\z}{r}\right),\quad 0<r<1,\quad z\in\D\setminus\{\vp(0)\}.
    $$
The classical Nevanlinna  counting function
    $
    N_\vp(z)=\sum_{\z\in\vp^{-1}(z)}\log\frac{1}{|\z|}
    $
is not subharmonic, but the partial Nevanlinna counting functions
    $
    N_\vp(r,z)=\sum_{\z\in\vp^{-1}(z),|\z|\le r}\log\frac{r}{|\z|}
    $
are in the set $\D\setminus\{\vp(0)\}$. This fact implies the submean value property of $N_\vp$~\cite{ShapiroAnnals87}, and allows us to show that the generalized Nevanlinna counting function is subharmonic on $\D\setminus\{\vp(0)\}$.

\begin{lemma}\label{Nsubharmonic}
Let $\vp$ be an
analytic self-map of $\D$ and $v$ a radial weight. Then
    \begin{equation}\label{eq:formulaN}
    N_{\vp,v^\star}(z)=\int_0^1 N_\vp(s,z)v(s)\,ds=\int_{|\z(z)|}^1\frac{n(r,z)}{r}\widehat{v}(r)\,dr,\quad z\ne\vp(0),
    \end{equation}
where $\z(z)$ is a preimage of $z$ with the minimum modulus. In particular, $N_{\vp,v^\star}$ is subharmonic on $\D\setminus\{\vp(0)\}$.
\end{lemma}

\begin{proof}
We borrow the argument of proof from \cite[Proposition~6.6]{ShapiroAnnals87}. For simplicity, write $n(r)=n(r,z)$.
An integration by parts yields
    \begin{equation}\label{111111}
    \begin{split}
    N_{\vp,v\star}(r,z)&=\int_0^r v^\star\left(\frac{t}{r}\right)\,dn(t)=\int_0^r\left(\int_{\frac{t}{r}}^1v(s)\,ds\right)\frac{n(t)}{t}\,dt,\quad z\neq \vp(0),
    \end{split}
    \end{equation}
and similarly,
    \begin{equation}\label{24}
    N_\vp(r,z)=\int_0^r \frac{n(t)}{t}\,dt,\quad z\neq \vp(0).
    \end{equation}
Another integration by parts in \eqref{111111} and \eqref{24} give
    \begin{equation*}
    \begin{split}
    N_{\vp,v^\star}(r,z)=\frac{1}{r}\int_0^r N_\vp(t,z)v\left(\frac{t}{r}\right)\,dt=\int_0^1 N_\vp(rs,z)v(s)\,ds.
    \end{split}
    \end{equation*}
By letting $r\to 1^-$ and using the monotone convergence theorem, the first equality in \eqref{eq:formulaN} follows. Since the partial counting functions $N_\vp(s,z)$ are subharmonic on $\D\setminus\{\vp(0)\}$, it follows from what we just proved that $N_{\vp,v^\star}$ has the submean value property.
Moreover, the partial counting functions are continuous on $\D\setminus\{\vp(0)\}$, 
and hence Littlewood's inequality \cite[p.~187]{Shapiro93}
    \begin{equation}\label{Eq:LittlewoodInequality}
    N_\vp(z)\le\log\frac{1}{|\vp_z(\vp(0))|},\quad
    \vp_z(u)=\frac{z-u}{1-\overline{z}u},\quad z,u\in\D,
    \end{equation}
the dominated convergence theorem and the first equality in \eqref{eq:formulaN} show that $N_{\vp,v^\star}$ is continuous as well. Consequently, $N_{\vp,v^\star}$ is subharmonic on $\D\setminus\{\vp(0)\}$.

To see the second equality in \eqref{eq:formulaN}, note first that $n(r,z)=0$ for all $r<|\z(z)|$. This and an integration by parts yield
    \begin{equation*}
    \begin{split}
    N_{\vp,v^\star}(z)&=\int_0^1N_\vp(r,z)v(r)\,dr=\int_{|\z(z)|}^1\int_0^r\frac{n(t,z)}{t}\,dt\,v(r)\,dr=\int_{|\z(z)|}^1\frac{n(t,z)}{t}\widehat{v}(t)\,dt,
    \end{split}
    \end{equation*}
and we are done.
\end{proof}

Theorem~\ref{Theorem:introduction-bounded-composition-operators} is contained in the following result.

\begin{theorem}\label{Theorem:bounded-composition-operators}
Let $0<p,q<\infty$, $\omega\in\DD$ and $v$ be a radial weight, and let $\vp$ be an
analytic self-map of $\D$.
\begin{itemize}
\item[\rm(a)] If $p>q$, then the following assertions are equivalent:
\begin{enumerate}
\item[\rm(i)] $C_\vp:A^p_\om\to A^q_v$ is bounded;
\item[\rm(ii)] $C_\vp:A^p_\om\to A^q_v$ is compact;
\item[\rm(iii)] $\displaystyle \int_\D\left(\sup_{z\in\Gamma(\z)}\frac{\int_{\Delta(z,r)}N_{\vp,v^\star}(u)\,dA(u)}{\om(S(z))(1-|z|)^2}\right)^\frac{p}{p-q}\om(\z)\,dA(\z)<\infty$ for any fixed $r\in(0,1)$;
    \item[\rm(iv)] $N\left(\frac{N_{\vp,v^\star}}{\om^\star}\right)\in L^\frac{p}{p-q}_\om$.
\end{enumerate}
\item[\rm(b)] If $q\ge p$, then the following assertions are equivalent:
\begin{enumerate}
\item[\rm(i)] $C_\vp:A^p_\om\to A^q_v$ is bounded;
\item[\rm(ii)] $\displaystyle z\mapsto \frac{\int_{\Delta(z,r)}N_{\vp,v^\star}(\z)\,dA(\z)}{\om(S(z))^\frac{q}{p}(1-|z|)^2}\in L^\infty$ for any fixed $r\in(0,1)$;
\item[\rm(iii)] $\displaystyle z\mapsto\frac{\int_{S(z)}N_{\vp,v^\star}(\z)\,dA(\z)}{\om(S(z))^\frac{q}{p}(1-|z|)^2}\in L^\infty$;
\item[\rm(iv)] $\displaystyle \limsup_{|z|\to1^-}\frac{N_{\vp,v^\star}(z)}{\om^\star(z)^\frac{q}{p}}<\infty$;
\item[\rm(v)] There exists $\eta=\eta(\om)>1$ such that
    $$
    \sup_{a\in\D}\int_\D\left(\frac{1}{\om(S(a))}\left(\frac{1-|a|}{|1-\overline{a}\vp(z)|}\right)^\eta\right)^\frac{q}{p}v(z)\,dA(z)<\infty;
    $$
\item[\rm(vi)] There exists $\eta=\eta(\om)>1$ such that
    $$
    \sup_{a\in\D}
    \int_\D\left(\frac{1}{\omega(S(a))}\frac{(1-|a|)^{\eta}}{|1-\overline{a}z|^{\eta+\frac{2p}{q}}}\right)^\frac{q}{p}N_{\vp,v^\star}(z)\,dA(z)<\infty.
    $$
\end{enumerate}
\item[\rm(c)] If $q\ge p$, then the following assertions are equivalent:
\begin{enumerate}
\item[\rm(i)] $C_\vp:A^p_\om\to A^q_v$ is compact;
\item[\rm(ii)] $\displaystyle \lim_{|z|\to1^-}\frac{\int_{\Delta(z,r)}N_{\vp,v^\star}(\z)\,dA(\z)}{\om(S(z))^\frac{q}{p}(1-|z|)^2}=0$ for any fixed $r\in(0,1)$;
\item[\rm(iii)] $\displaystyle \lim_{|z|\to1^-}\frac{\int_{S(z)}N_{\vp,v^\star}(\z)\,dA(\z)}{\om(S(z))^\frac{q}{p}(1-|z|)^2}=0$;
\item[\rm(iv)] $\displaystyle \lim_{|z|\to1^-}\frac{N_{\vp,v^\star}(z)}{\om^\star(z)^\frac{q}{p}}=0$;
\item[\rm(v)] There exists $\eta=\eta(\om)>1$ such that
    $$
    \lim_{|a|\to1^-}\int_\D\left(\frac{1}{\om(S(a))}\left(\frac{1-|a|}{|1-\overline{a}\vp(z)|}\right)^\eta\right)^\frac{q}{p}v(z)\,dA(z)=0;
    $$
\item[\rm(vi)] There exists $\eta=\eta(\om)>1$ such that
    $$
    \lim_{|a|\to1^-}
    \int_\D\left(\frac{1}{\omega(S(a))}\frac{(1-|a|)^{\eta}}{|1-\overline{a}z|^{\eta+\frac{2p}{q}}}\right)^\frac{q}{p}N_{\vp,v^\star}(z)\,dA(z)=0.
    $$
\end{enumerate}
\end{itemize}
Moreover, the condition {\rm (v)} in both {\rm (b)} and {\rm (c)} characterizes bounded and compact operators $C_\vp:A^p_\om\to A^q_v$, respectively, if $v$ is only assumed to be a weight.
\end{theorem}

\begin{proof}
A change of variable and \cite[Theorem~1]{PelRatMathAnn} show that $C_\vp:A^p_\om\to A^q_v$ is bounded if and only if $C_\vp:A^{sp}_\om\to A^{sq}_v$ is bounded for $s>0$. By choosing $s=2/q$ and using \eqref{LP1}, this is in turn equivalent to
    $$
    4\int_\D|f'(\vp(z))|^2|\vp'(z)|^2v^\star(z)\,dA(z)+v(\D)|f(\vp(0))|^2\le C\|f\|_{A^{\frac{2p}{q}}_\om}^{\frac{q}{p}}
    $$
for all $f\in A^{\frac{2p}{q}}_\om$. Another change of variable on the left and \cite[Theorem~2]{PelRatMathAnn}, with $n=1$, now shows that (i) and (iii) in (a), and (i)--(iii) in (b), are equivalent.

With Lemma~\ref{Lemma:replacement-Lemmas-Memoirs} in hand, one readily sees that \cite[Lemma~5.3]{PelRat} holds for $\om\in\DD$. This observation together with \cite[Theorem~1]{PelRatMathAnn} and standard arguments show that (i) and (v) in (b) are equivalent. An application of \eqref{LP1} and a change of variable show that (v) and (vi) in (b) are equivalent. Clearly, (iv) implies (ii) in (b) by  Lemma~\ref{Lemma:replacement-Lemmas-Memoirs}, and the opposite implication follows by Lemma~\ref{Nsubharmonic}.

To prove (c) we first observe that the main argument in the proof of \cite[Theorem~2.1(ii)]{PelRat} combined with \cite[Theorem~1]{PelRatMathAnn} shows that, for $q\ge p$, $\om\in\DD$ and a positive Borel measure on $\D$, the identity operator $I_d:A^p_\om\to L^q(\mu)$ is compact if and only if $\lim_{|I|\to0}\frac{\mu(S(I))}{(\om(S(I)))^\frac{q}{p}}=0$. A reasoning similar to that in the previous paragraph now implies the equivalence of (i) and (v) in (c). A change of variable together with \eqref{LP1} now gives (v)$\Leftrightarrow$(vi), and straightforward calculations show that (iv) implies any of the conditions (ii), (iii) and (vi), see the proof of Theorem~\ref{Thm:EssentialNormBergman} for details of similar calculations. The necessity of (iv) for any of these three conditions to be satisfied follows by Lemma~\ref{Nsubharmonic}.

By Lemma~\ref{Nsubharmonic}, (iii) implies (iv) in (a). To complete the proof it remains to show that $C_\vp:A^p_\om\to A^q_v$ is compact if (iv) in (a) is satisfied. To see this, it suffices to prove that for each uniformly bounded sequence $\{f_n\}$ in $A^p_\om$, that converges to $0$ uniformly on compact subsets of $\D$, we have $\|C_\vp(f_n)\|_{A^q_v}\to0$, as $n\to\infty$. Let $\ep>0$ and choose $r_0\in\left(\frac12,1\right)$ such that $\vp(0)\in
D(0,2r_0-1)$ and
   \begin{equation}\label{6}
   \left(\int_{r_0\le |z|<1}\left(\sup_{z\in\Gamma(\z)}\frac{N_{\vp,v^\star}(z)}{\om(S(z))}\right)^\frac{p}{p-q}\om(\z)\,dA(\z)\right)^{(p-q)/p}<\ep.
   \end{equation}
Further, let $n_0\in\N$ such that $|f_n(z)|<\ep^{1/q}$ for all
$n\ge n_0$ and $z\in\overline{D(0,r_0)}$. Then \cite[Theorem~4.2]{PelRat} and a change of variable show that
    \begin{equation*}
    \begin{split}
    \|C_\vp(f_n)\|_{A^q_{v}}^q
    &\asymp\int_\D\Delta|f_n|^q(z)N_{\vp,v^\star}(z)\,dA(z)+v(\D)|f_n(\vp(0))|^q\\
    &\lesssim\int_\D\left(\int_{\Gamma(\z)}\Delta|f_n|^q(z)N_{\vp,v^\star}(z)\frac{\,dA(z)}{\om(T(z))}\right)\om(\z)\,dA(\z)+\e v(\D)
    \end{split}
    \end{equation*}
for all $n\ge n_0$. Split now the outer integral into two pieces. By the fact $\om(S(z))\asymp \om(T(z))$, H\"older's inequality, \eqref{6} and \cite[Theorem~1.3]{Pavlovic2013},
    \begin{equation}
    \begin{split}\label{th1:1}
 &\int_{\{r_0\le |\z|<1\}}\left(\int_{\Gamma(\z)}\Delta|f_n|^q(z)N_{\vp,v^\star}(z)\frac{\,dA(z)}{\om(T(z))}\right)\om(\z)\,dA(\z)
  \\ & \lesssim \int_{\{r_0\le |\z|<1\}} \left(\sup_{z\in\Gamma(\z)}\frac{N_{\vp,v^\star}(z)}{\om(S(z))}\right)
  \left(\int_{\Gamma(\z)}\Delta|f_n|^q(z)\,dA(z)\right)\om(\z)\,dA(\z)
   \\ & \lesssim \e\left(\int_\D\left(\int_{\Gamma(\z)}\Delta|f_n|^q(z)\,dA(z)\right)^\frac{p}{q}\om(\z)\,dA(\z)\right)^\frac{q}{p}
    \lesssim\e\|f_n\|_{A^p_\om}^q\lesssim\e.
    \end{split}
    \end{equation}
As $|f_n|<\ep^{1/q}$ in $\overline{D(0,r_0)}$ for all $n\ge n_0$, a reasoning similar to that above yields
 \begin{equation*}
    \begin{split}
 &\int_{D(0,r_0)}\left(\int_{\Gamma(\z)}\Delta|f_n|^q(z)N_{\vp,v^\star}(z)\frac{\,dA(z)}{\om(T(z))}\right)\om(\z)\,dA(\z)
 \\ & \lesssim\left(\int_{D(0,r_0)}\left(\sup_{z\in\Gamma(\z)}\frac{N_{\vp,v^\star}(z)}{\om(S(z))}\right)^\frac{p}{p-q}\om(\z)\,dA(\z)\right)^\frac{p-q}{p}
 \\ & \quad\cdot \left(\int_{D(0,r_0)}\left(\int_{\Gamma(\z)}\Delta|f_n|^q(z)\,dA(z)\right)^\frac{p}{q}\om(\z)\,dA(\z)\right)^\frac{q}{p}
 \\ & \lesssim \left(\int_{0}^{r_0}\int_{\T}\left(\int_{\Gamma(e^{i\theta})}\Delta|(f_n)_s|^q(z)\,dA(z)\right)^\frac{p}{q}\,d\theta\,\om(s)\,ds\right)^\frac{q}{p}
  \\ & \lesssim \left(\int_{0}^{r_0} M^p_p(s,f_n)\,\om(s)\,ds\right)^\frac{q}{p}\lesssim \ep,
    \end{split}
    \end{equation*}
which together with \eqref{th1:1} gives $\|C_\vp(f_n)\|_{A^q_v}\to0$, as $n\to\infty$. This completes the proof.
\end{proof}

Each composition operator is bounded on the Hardy and the standard weighted Bergman spaces by Littlewood's subordination principle. We will use the following result to see that the same happens also in $A^p_\om$ when $\om\in\DD$. The lemma also implies that $\om^\star$ is subharmonic in $\D\setminus\{0\}$.

\begin{lemma}\label{LemmaLittlewoodTypeInequality}
Let $\vp$ be an analytic self-map of $\D$ and $\omega$ a radial weight. Then
    $$
    N_{\vp,\om^\star}(z)\le \omega^\star\left(\vp_z(\vp(0))\right),\quad z\in\D\setminus\{\vp(0)\}.
    $$
\end{lemma}

\begin{proof}
By using Littlewood's inequality \eqref{Eq:LittlewoodInequality} and the monotonicity of $\omega^\star$, we obtain
    \begin{equation*}
    \begin{split}
    N_{\vp,\om^\star}(z)&=\sum_{n}\omega^\star(\z_n)
    \le\omega^\star\left(\prod_n \z_n\right)=\omega^\star\left(\exp\left(-\sum_n\log\frac1{|\z_n|}\right)\right)\\
    &\le\omega^\star\left(\vp_z(\vp(0))\right),\quad z\in\D\setminus\{\vp(0)\},
    \end{split}
    \end{equation*}
for any radial weight $\om$.
\end{proof}

If $\omega\in\DD$, Lemma~\ref{LemmaLittlewoodTypeInequality} together with Lemma~\ref{Lemma:replacement-Lemmas-Memoirs} gives
    \begin{eqnarray*}
    \limsup_{|z|\to1^-}\frac{N_{\vp,\om^\star}(z)}{\widehat{\om}(z)(1-|z|)}
    &\lesssim&\limsup_{|z|\to1^-}\frac{(1-|\vp_{\vp(0)}(z)|^2)\widehat{\om}\left(\vp_{\vp(0)}(z)\right)}
    {(1-|z|^2)\widehat{\omega}(z)}\\
    &\lesssim&\limsup_{|z|\to1^-}\left(\frac{(1-|\vp_{\vp(0)}(z)|^2)}{1-|z|^2}\right)^{\b+1}
    \le\left(\frac{1+|\vp(0)|}{1-|\vp(0)|}\right)^{\b+1}.
    \end{eqnarray*}
Therefore there exist constants $\eta=\eta(\omega)>1$ and
$C=C(\eta,\omega)>0$ such that
    \begin{equation}\label{90}
    N_{\vp,\om^\star}(z)\le C\left(\frac{1+|\vp(0)|}{1-|\vp(0)|}\right)^\eta\om^\star(z),\quad |z|\to1^-,
    \end{equation}
for all analytic self-maps $\vp$. In view of Theorem~\ref{Theorem:bounded-composition-operators}(b), this shows that each composition operator is bounded on the Bergman space $A^p_\om$ when $\om\in\DD$. It is worth noticing that the validity of \eqref{90} for all analytic self-maps implies $\om\in\DD$ as is seen by choosing $\vp(z)=(1+z)/2$. The general phenomenon behind this is described in the following result.

\begin{proposition}
Let $\om$ be a radial weight. Then the following conditions are equivalent:
\begin{enumerate}
\item[\rm(i)] $\om\in\DD$;
\item[\rm(ii)] For any analytic self-map $\vp$ of $\D$, there exist constants $M=M(\vp,\om)>0$ and $r_0=r_0(\vp)\in(0,1)$ such that
    $$
    N_{\vp,\om^\star}(z)\le  M\om^\star(z),\quad |z|\ge r_0;
    $$
\item[\rm(iii)] There exists an automorphism $\vp_a(z)=\frac{a-z}{1-\overline{a}z}$ of $\D$ and constants $M=M(a,\om)>0$ and $r_0=r_0(a)\in(0,1)$ such that
    $$
    N_{\vp_{a},\om^\star}(z)\le  M\om^\star(z),\quad |z|\ge r_0.
    $$
\end{enumerate}
\end{proposition}

\begin{proof} By \eqref{90} it suffices to show that (iii)$\Rightarrow$(i). If (iii) is satisfied,
then Lemma~\ref{LemmaLittlewoodTypeInequality} yields $N_{\vp,\om^\star}(z)\le\om^\star(\vp_{a}(z))\le M\om^\star(z)$ for any analytic self-map of $\D$ such that $\vp(0)=a$. Take $\vp(z)=a+(1-|a|)z$, and choose $m=m(|a|)>0$ such that $\frac{m}{m+1}<|a|$. Then the inequalities $1-t\le-\log t\le\frac1t(1-t)$ give
    \begin{equation*}
    \begin{split}
    \frac{m}{m+1}\frac{1-s}{1-|a|}\widehat{\om}\left(\frac{m+\frac{s-|a|}{1-|a|}}{m+1}\right)
    &\le\int_{\frac{m+\frac{s-|a|}{1-|a|}}{m+1}}^1\left(r-\frac{s-|a|}{1-|a|}\right)\om(r)\,dr\\
    &\le\int_{\frac{s-|a|}{1-|a|}}^1\left(r-\frac{s-|a|}{1-|a|}\right)\om(r)\,dr\\
    &\le\om^\star\left(\frac{s-|a|}{1-|a|}\right)\le M\om^\star(s)
    \le M\frac{1-s}{s}\widehat{\om}(s),\quad s>|a|,
    \end{split}
    \end{equation*}
and hence
    \begin{equation}\label{eq:estimate111}
    \begin{split}
    \widehat{\om}\left(\frac{m+\frac{s-|a|}{1-|a|}}{m+1}\right)
    \le\frac{(1-|a|)(m+1)M}{ms}\widehat{\om}(s),\quad s>\max\{r_0,|a|\}.
    \end{split}
    \end{equation}
Since
    $$
    1-\frac{m+\frac{s-|a|}{1-|a|}}{m+1}=\frac{1-s}{(m+1)(1-|a|)}
    $$
and $(m+1)(1-|a|)<1$, we deduce $\om\in\DD$ by applying \eqref{eq:estimate111} sufficiently many times.
\end{proof}

The next result offers an asymptotic formula for the operator norm of $C_\vp:A^p_\om\to A^p_\om$ when $\om$ belongs to a certain subclass of $\DD$.

\begin{proposition}\label{co:norma}
Let $0<p<\infty$ and $\om\in\DD$, and let $\vp$ be an analytic self-map of~$\D$. Then each $C_\vp:A^p_\om\to A^p_\om$ is bounded and there exists $c=c(\om,p)>0$ such that
    \begin{equation}\label{lower}
    \frac{c}{\widehat{\omega}(\vp(0))(1-|\vp(0)|)}\le\|C_\vp\|_{(A^p_\om\to A^p_\om)}^p.
    \end{equation}
Moreover, there exists $C=C(\om,p)>0$ such that
    $$
    \|C_\vp\|_{(A^p_\om\to A^p_\om)}^p\le\frac{C}{\widehat{\omega}(\vp(0))(1-|\vp(0)|)}
    $$
for all analytic self-maps $\vp$, if and only if, there exists $M=M(\om)>0$ such that
    \begin{equation}\label{extra}
    \widehat{\om}\left(\vp_{t}(r)\right)\le M\frac{\widehat{\om}(t)}{\widehat{\om}(r)},\quad 0\le r\le t<1.
    \end{equation}
\end{proposition}

\begin{proof}
Each $C_\vp$ is bounded on~$A^p_\omega$ if $\om\in\DD$ by Theorem~\ref{Theorem:bounded-composition-operators}(b) and \eqref{90}. If $f\in A^p_{\om}$ and $\om$ is a radial weight, then
    $$
    \|f\|_{A^p_\omega}^p\ge\int_{\D\setminus
    D(0,r)}|f(z)|^p\omega(z)\,dA(z)\gtrsim
    M_p^p(r,f)\widehat{\om}(r),\quad r\ge\frac12,
    $$
and hence the well known relation $M_\infty(f,r)\lesssim
M_p(\frac{1+r}{2},f)(1-r)^{-1/p}$ and the hypotheses yield
    \begin{equation}\label{Eq:radial-growth}
    M_\infty(r,f)\lesssim\|f\|_{A^p_\om}\left(\frac{1}{\widehat{\om}(r)(1-r)}\right)^{1/p},\quad r\ge\frac12.
    \end{equation}
This together with the test functions $f_a(z)=((1-|a|)/(1-\overline{a}z))^\frac{\g+1}{p}$, where $\g=\g(\om)>0$ is sufficiently large, and Lemma~\ref{Lemma:replacement-Lemmas-Memoirs} shows that each point evaluation $L_a(f)=f(a)$ is a bounded
linear functional on $A^p_\om$ with $\|L_a\|^p\asymp(\widehat{\omega}(a)(1-|a|))^{-1}$.
Therefore the identity $C_\vp^\star(L_0)=L_{\vp(0)}$ gives
    $$
    \|C_\vp\|_{A^p_\om\to A^p_\om}^p=\|C^\star_\vp\|^p\ge\frac{\|C^\star_\vp(L_0)\|^p}{\|L_0\|^p}
    \asymp\|L_{\vp(0)}\|^p\asymp\frac{1}{\widehat{\omega}(\vp(0))(1-|\vp(0)|)},
    $$
which is the desired lower estimate~\eqref{lower}.

Littlewood's subordination principle shows that $\|f\circ\vp\|_{A^p_\om}\le\|f\circ\vp_{\vp(0)}\|_{A^p_\om}$. Further, by \cite[Theorem~1.3]{Pavlovic2013} and the proof of \cite[Theorem~4.2]{PelRat},
    \begin{equation*}
    \|f\circ\vp_{\vp(0)}\|_{A^p_\om}^p\asymp\int_\D\Delta|f\circ\vp_{\vp(0)}|^p(z)\int_{|z|}^1\om(r)\left(1-\frac{|z|}{r}\right)r\,dr\,dA(z)+|f(\vp(0))|^p,
    \end{equation*}
where the integral over $\D$ on the right is dominated by
    \begin{equation*}
    \left(\int_{D(0,|\vp(0)|)}+\int_{\D\setminus D(0,|\vp(0)|)}\right)\Delta|f|^p(z)(1-|\vp_{\vp(0)}(z)|^2)\int_{|\vp_{\vp(0)}(z)|}^1\om(r)r\,dr\,dA(z)=I_1+I_2.
    \end{equation*}
Now
    \begin{equation*}
    I_1\le\frac{4\int_0^1\om(r)r\,dr}{\widehat{\om}(\vp(0))(1-|\vp(0)|)}\int_\D\Delta|f|^p(z)(1-|z|)\widehat{\om}(z)\,dA(z)
    \lesssim\frac{\|f\|_{A^p_\om}^p}{\widehat{\om}(\vp(0))(1-|\vp(0)|)}
    \end{equation*}
by \cite[Theorem~4.2]{PelRat}, and the hypothesis yields the same upper estimate for $I_2$. Since the point evaluation functional $L_{\vp(0)}$ is bounded on $A^p_\om$ with $\|L_{\vp(0)}\|_{(A^p_\om)^\star}^p\asymp(\widehat{\omega}(\vp(0))(1-|\vp(0)|))^{-1}$ by the first part of the proof, the upper bound follows.

Conversely, take $0<b\le a<1$ and consider $\vp=\vp_b$. The hypothesis together with the proof of Theorem~\ref{Theorem:bounded-composition-operators} show that
    \begin{equation*}
    \begin{split}
    \frac{C}{\widehat{\om}(b)(1-|b|)}&\ge\|C_{\vp_b}\|_{A^p_\om\to A^p_\om}^p\gtrsim\frac{\int_{\Delta(a,r)}N_{\vp_b,\om^\star}(z)\,dA(z)}{\om(S(a))(1-|a|)^2}
    \gtrsim\frac{\int_{\Delta(a,r)}(1-|\vp_b(z)|)\int_{|\vp_b(z)|}^1\om(s)\,ds\,dA(z)}{\widehat{\om}(a)(1-|a|)^3}\\
    &\asymp\frac{(1-|\vp_a(b)|)\int_{|\vp_a(b)|}^1\om(s)\,ds}{\widehat{\om}(a)(1-|a|)}\gtrsim
    \frac{(1-|b|)}{(1-ab)^2}\frac{\widehat{\om}(\vp_a(b))}{\widehat{\om}(a)},
    \end{split}
    \end{equation*}
and \eqref{extra} follows.
\end{proof}

Straightforward calculations show that the standard weights as well as the weights $v_\a(z)=(1-|z|)^{-1}\left(\log\frac1{1-|z|}\right)^{-\a}$ satisfy~\eqref{extra}. Moreover, the choice $t=\frac{1+r}{2}$ shows that each radial weight satisfying \eqref{extra} belongs to $\DD$. The converse is not true in general as is seen by considering the weight $\om(r)=\log\frac{e}{1-r}-1$.

Proposition~\ref{co:norma} should be compared with \cite[Theorem~6.11]{Zhu}. Moreover, it is worth mentioning that there exist radial weights $\om$ and analytic self-maps $\vp$ such that $C_\vp$ is not bounded on $A^p_\omega$~\cite{KrMc}.

Theorem~\ref{Theorem:bounded-composition-operators}(c) suggests that the quantities appearing there are actually asymptotic formulas for the essential norm of a bounded operator $C_\vp: A^p_\om\to A^q_v$. Our next goal is to show that this is indeed the case. If $0<q,p<\infty$ and $T: A^p_\om\to A^q_\om$ is bounded, then the essential norm of $T$ is defined as
    $$
    \|T\|_{e}=\inf_{K}\sup_{\|f\|_{A^p_\om}=1}\|(T-K)(f)\|_{A^q_v}=\inf_{K}\|T-K\|_{(A^p_\om\to A^q_v)},
    $$
where the infimum is taken over all compact operators $K:A^p_\om\to A^q_v$. The key step in the proof of the following result is to show that the essential norm is comparable to the quantity $D$.

\begin{theorem}\label{Thm:EssentialNormBergman}
Let $0<p\le q<\infty$, $\omega\in\DD$ and $v$ be a radial weight. Let $\vp$ be an
analytic self-map of $\D$ such that $C_{\vp}: A^p_\om\to A^q_v$ is bounded. Then there exists $\eta=\eta(\om)>1$ such that, for each fixed $r\in(0,1)$, the following quantities are comparable:
\begin{itemize}
\item[] $\|C_\vp\|^q_e$;
\item[] $\displaystyle A=\limsup_{|z|\to1^-}\frac{\int_{\Delta(z,r)}N_{\vp,v^\star}(\z)\,dA(\z)}{\om(S(z))^\frac{q}{p}(1-|z|)^2}$;
\item[] $\displaystyle B=\limsup_{|z|\to1^-}\frac{\int_{S(z)}N_{\vp,v^\star}(\z)\,dA(\z)}{\om(S(z))^\frac{q}{p}(1-|z|)^2}$;
\item[] $\displaystyle C=\limsup_{|z|\to1^-}\frac{N_{\vp,v^\star}(z)}{\om^\star(z)^\frac{q}{p}}$;
\item[] $\displaystyle
    D=\limsup_{|a|\to1^-}\int_\D\left(\frac{1}{\om(S(a))}\left(\frac{1-|a|}{|1-\overline{a}\vp(z)|}\right)^\eta\right)^\frac{q}{p}v(z)\,dA(z)$;
\item[] $\displaystyle
    E=\limsup_{|a|\to1^-}\int_\D\left(\frac{1}{\omega(S(a))}\frac{(1-|a|)^{\eta}}{|1-\overline{a}z|^{\eta+\frac{2p}{q}}}\right)^\frac{q}{p}N_{\vp,v^\star}(z)\,dA(z)$.
\end{itemize}
\end{theorem}

\begin{proof}
By employing an approach used in~\cite{CharArch12} with \cite[Theorem~9]{PelRatMathAnn}, the proof of \cite[Theorem~2.1(ii)]{PelRat} and standard techniques involving the test functions
    $$
    f_{a,p}(z)=\frac{1}{ \omega(S(a))^{\frac{1}{p}}}\left(\frac{1-|a|^2}{1-\overline{a}z}\right)^{\frac{\eta}{p}},\quad a\in\D,
    $$
and Lemma~\ref{Lemma:replacement-Lemmas-Memoirs}, one deduces $\|C_\vp\|^q_e\asymp D$ for a suitably chosen $\eta=\eta(\om)>1$.

Clearly, $A,B\lesssim E$ while $C\lesssim A,B$ follows by Lemmas~\ref{Nsubharmonic} and \ref{Lemma:replacement-Lemmas-Memoirs}, and \eqref{LP1} and a change of variable give $D\asymp E$. To complete the proof, it remains to establish $\|C_\vp\|^q_e\lesssim C$. To do this, let $\e>0$ and choose $r_\e\in(0,1)$ such that
    $
    \frac{N_{\vp,\om^\star}(z)}{v^\star(z)^\frac{q}{p}}\le C+\e
    $
for all $r_\e\le|z|<1$.
Choose $\eta$ sufficiently large such that $\|C_\vp\|^q_e\asymp E$ and
$\limsup_{|a|\to1^-}\frac{(1-|a|)^{\eta}}{\omega(S(a))}=0$.
Then Theorem~\ref{Theorem:bounded-composition-operators}(b) and Lemma~\ref{Lemma:replacement-Lemmas-Memoirs}(iv) yield
    \begin{equation*}
    \begin{split}
    \|C_\vp\|_e^q&\asymp\limsup_{|a|\to1^-}\frac{(1-|a|)^s}{(\omega(S(a)))^\frac{q}{p}}
    \int_\D\frac{N_{\vp,v^\star}(z)}{|1-\overline{a}z|^{s+2}}\,dA(z),\quad s=\frac{q\eta}{p},\\
    &\lesssim \limsup_{|a|\to1^-}\frac{(1-|a|)^s}{(\omega(S(a)))^\frac{q}{p}}
    \int_0^{r_\e}\frac{(\om^\star(r))^\frac{q}{p}}{(1-|a|r)^{s+1}}\,dr\\
    &\quad+(C+\e)\limsup_{|a|\to1^-}\frac{(1-|a|)^s}{(\omega(S(a)))^\frac{q}{p}}
    \int_{r_\e}^1\frac{(\om^\star(r))^\frac{q}{p}}{(1-|a|r)^{s+1}}\,dr\\
    &\lesssim\limsup_{|a|\to1^-}\frac{(1-|a|)^s}{(\omega(S(a)))^\frac{q}{p}(1-r_\e)^{s+1}}
    \int_0^1(\omega^\star(r))^\frac{q}{p}\,dr\\
    &\quad+(C+\e)\limsup_{|a|\to1^-}\frac{(1-|a|)^s}{(\omega(S(a)))^\frac{q}{p}}
    \int_{r_\e}^1\frac{(\omega^\star(r))^\frac{q}{p}}{(1-|a|r)^{s+1}}\,dr\\ &\lesssim(C+\e)\sup_{a\in\D}\left(\frac{(1-|a|)^{s-\frac{q}{p}}}{\widehat{\omega}(a)^\frac{q}{p}}
    \int_0^1\frac{\widehat{\omega}(r)^\frac{q}{p}}{(1-|a|r)^{s+1-\frac{q}{p}}}\,dr\right).
    \end{split}
    \end{equation*}
Clearly,
    \begin{equation*}
    \begin{split}
    \frac{(1-|a|)^{s-\frac{q}{p}}}{\widehat{\omega}(a)^\frac{q}{p}}
    \int_{|a|}^1\frac{\widehat{\omega}(r)^\frac{q}{p}}{(1-|a|r)^{s+1-\frac{q}{p}}}\,dr
    \le1.
    \end{split}
    \end{equation*}
Moreover, if $\eta$ is sufficiently large, we may apply Lemma~\ref{Lemma:replacement-Lemmas-Memoirs}(iii) to the weight $\widehat{\om}^\frac{q}{p}\in\DD$ to see that
    \begin{equation*}
    \begin{split}
    \frac{(1-|a|)^{s-\frac{q}{p}}}{\widehat{\omega}(a)^\frac{q}{p}}
    \int_0^{|a|}\frac{\widehat{\omega}(r)^\frac{q}{p}}{(1-|a|r)^{s+1-\frac{q}{p}}}\,dr
    &\le\frac{(1-|a|)^{s-\frac{q}{p}}}{\widehat{\omega}(a)^\frac{q}{p}}
    \int_0^{|a|}\frac{\widehat{\omega}(r)^\frac{q}{p}}{(1-r)^{s+1-\frac{q}{p}}}\,dr\lesssim
    \frac{\widehat{\widehat{\om}^\frac{q}{p}}(a)}{\widehat{\omega}(a)^\frac{q}{p}(1-|a|)}\le1.
   \end{split}
    \end{equation*}
It follows that $\|C_\vp\|^q_{e}\lesssim C$.
\end{proof}

We next study the connection between the compactness of $C_\vp:A^p_\omega\to A^p_\omega$, $\om\in\DD$, and the
existence of angular derivative on $\T$. By Theorem~\ref{Theorem:bounded-composition-operators}(c), the compactness of
$C_\vp:A^p_\omega\to A^p_\omega$ does not depend on $p$.

Kriete and MacCluer~\cite{KrMc} done an extensive study in this topic in the case of the Bergman space~$A^2_\om$ induced by a weight $\om$ that decreases to zero faster than any power of $1-|z|$, as $z$ approaches the boundary.

The compactness of $C_\vp$ on the classical weighted Bergman space $A^p_\a$ is known to be closely related to the existence on angular derivatives at $\T$~\cite[Theorem~3.22]{CowenMac95}. The next theorem is a generalization this result.

\begin{theorem}\label{th:compderR}
Let $0<p<\infty$, $\om\in\DD$ such that $\widehat{\om}(r)\ge C\widehat{\om}\left(\frac{1+r}{2}\right)$ for some $C=C(\om)>1$, and $\vp$ be an analytic self-map of $\D$. Then $C_\vp:A^p_\omega\to A^p_\omega$ is compact if and only if $\vp$ has no finite angular derivative at any point of $\T$.
\end{theorem}

\begin{proof}
We provide a proof based on ideas from \cite[Section~3.2]{CowenMac95}. If $C_\vp:A^p_\omega\to A^p_\omega$ is compact, then Theorem~\ref{Theorem:introduction-bounded-composition-operators}(c) and Lemma~\ref{Lemma:replacement-Lemmas-Memoirs} yield
    \begin{equation*}
    \lim_{|a|\to 1^-}\frac{(1-|a|^2)\widehat{\om}(a)}{(1-|\vp(a)|^2)\widehat{\om}(\vp(a))}=0.
    \end{equation*}
In particular, there exists $r_0\in (0,1)$ such that $|\vp(a)|\le |a|$ for all $a\in\D\setminus D(0,r_0)$. Therefore, by Lemma~\ref{Lemma:replacement-Lemmas-Memoirs}, there exists $\beta=\beta(\om)>0$ and $C=C(\beta)$ such that
    \begin{equation*}
    0=\lim_{|a|\to 1^-}\frac{(1-|a|^2)\widehat{\om}(a)}{(1-|\vp(a)|^2)\widehat{\om}(\vp(a))}\ge C \lim_{|a|\to 1^-}
    \left(\frac{1-|a|^2}{1-|\vp(a)|^2}\right)^{1+\beta}.
    \end{equation*}
Consequently, the Julia-Carath\'eodory theorem (\cite[Theorem~2.44]{CowenMac95} or \cite[p.~57]{Shapiro93}) ensures that
$\vp$ has no finite angular derivative at any point of $\T$.

Assume next that $\vp$ has no finite angular derivative at any point of $\T$. For $a\in\vp(\D)$, let $z(a)$ be one of the points in $\vp^{-1}(a)$ with minimum modulus. Now, pick up a sequence $\{a_n\}$ such that $\lim_{n\to\infty}|a_n|=1$ and
    \begin{equation}\label{eq:j2} \limsup_{|a|\to1^-}\frac{N_{\vp,\om^\star}(a)}{\om^\star(a)}=\lim_{n\to\infty}\frac{N_{\vp,\om^\star}(a_n)}{\om^\star(a_n)}.
    \end{equation}
Lemma~\ref{Lemma:replacement-Lemmas-Memoirs} gives
    \begin{equation}\label{2222222}
    \om^\star(z)\asymp\widehat{\om}(z)\log\frac{1}{|z|},\quad \frac{1}{2}\le |z|<1.
    \end{equation}
Hence
    $$
    N_{\vp,\om^\star}(a_n)\asymp \sum_{z\in\vp^{-1}(a_n)}\widehat{\om}(z)\log\frac{1}{|z|}
    \le \widehat{\om}(z(a_n))N_{\vp}(a_n)
    $$
and further,
    \begin{equation}\label{eq:j1}
    \begin{split}
    \lim_{n\to\infty}\frac{N_{\vp,\om^\star}(a_n)}{\om^\star(a_n)}
    &\lesssim \lim_{n\to\infty}\frac{\widehat{\om}(z(a_n))}{\widehat{\om}(a_n)}\frac{N_{\vp}(a_n)}{\log\frac{1}{|a_n|}}
    \lesssim \lim_{n\to\infty}\frac{\widehat{\om}(z(a_n))}{\widehat{\om}(a_n)},
    \end{split}
    \end{equation}
where the last inequality follows by \cite[Corollary on p.~188]{Shapiro93}. On the other hand, since $\vp$ has no finite angular derivative at any point of $\T$, $\lim_{|z|\to 1^-} \frac{1-|z|}{1-|\vp(z)|}=0$. Thus
    \begin{equation}\label{1111111}
    \lim_{n\to\infty}\frac{1-|z(a_n)|}{1-|a_n|}=0,
    \end{equation}
and, in particular, there exists $n_0\in\N$ such that $|z(a_n)|\ge |a_n|$ for all $n\ge n_0$. A reasoning similar to the proof of \cite[Lemma~1]{PRAntequera} (compare with Lemma~\ref{Lemma:replacement-Lemmas-Memoirs}) shows that the hypothesis $\widehat{\om}(r)\ge C\widehat{\om}\left(\frac{1+r}{2}\right)$ for $C>1$ implies the existence of $c=c(\om)>0$ and $\a=\a(\om)>0$ such that
    \begin{equation}\label{reverse}
    \widehat{\om}(r)\ge c\left(\frac{1-r}{1-t}\right)^\a\widehat{\om}(t),\quad 0\le r\le t<1.
    \end{equation}
This together with \eqref{eq:j1} and \eqref{1111111} gives
    \begin{equation*}
    \lim_{n\to\infty}\frac{N_{\vp,\om^\star}(a_n)}{\om^\star(a_n)}
    \lesssim \lim_{n\to\infty}\left(\frac{1-|z(a_n)|}{1-|a_n|}\right)^\alpha=0.
    \end{equation*}
Finally, by \eqref{eq:j2} and Theorem~\ref{Theorem:introduction-bounded-composition-operators}(c), we deduce that $C_\vp:A^p_\omega\to A^p_\omega$ is compact.
\end{proof}

By combining the arguments from the proof of Theorem~\ref{th:compderR} and \cite[Corollary~3.21]{CowenMac95} we obtain the following result.

\begin{corollary}\label{CorAngularDerivativeBoundedValenceI}
Let $0<p<\infty$, $\om\in\DD$ and $\vp$ be a bounded valent
analytic self-map of $\D$. Then $C_\vp:A^p_\omega\to A^p_\omega$
is compact if and only if $\vp$ does not have finite angular
derivative at any point of $\T$.
\end{corollary}

Another immediate consequence of the proof of Theorem~\ref{th:compderR} is a sufficient condition for $C_\vp:A^p_\omega\to A^p_\omega$ to be compact.

\begin{corollary}\label{CorCompactI}
Let $0<p<\infty$, $\om\in\DD$ and $\vp$ be an analytic self-map of $\D$. If
    \begin{equation}\label{113}
    \lim_{|z|\to1}\frac{\widehat{\om}(z)}{\widehat{\om}(\vp(z))}=0,
    \end{equation}
then $C_\vp:A^p_\omega\to A^p_\omega$ is compact.
\end{corollary}

Consider the univalent lens map $\vp(z)=1-(1-z)^\g$, where $0<\g<1$. Clearly,
    $$
    \frac{1-|z|^2}{1-|\vp(z)|^2}\asymp\frac{1-|z|}{|1-\vp(z)|}=\frac{1-|z|}{|1-z|^\gamma}\to0,\quad|z|\to1^-,
    $$
and therefore $C_\vp:A^p_\omega\to A^p_\omega$ is compact for all $0<p<\infty$ and $\om\in\DD$ by Corollary~\ref{CorAngularDerivativeBoundedValenceI}. However, for each $\a>1$ we have
    $$
    \frac{\widehat{v_\a}(r)}{\widehat{v_\a}(\vp(r))}\asymp\left(\frac{\log\frac{e}{1-\vp(r)}}{\log\frac{e}{1-r}}\right)^{\a-1}\to\g^{\a-1}>0,\quad r\to1^-,
    $$
and hence \eqref{113} fails. This shows that \eqref{113} does not characterize compact operators $C_\vp:A^p_\omega\to A^p_\omega$ when $\om\in\DD$. The last lemma clarifies the general situation with regard to \eqref{113}. The proof follows by the Julia-Carath\'eodory theorem, Lemma~\ref{Lemma:replacement-Lemmas-Memoirs} and \eqref{2222222}.

\begin{lemma}
Let $\vp$ be an analytic self-map of $\D$ and $\om\in\DD$. Consider the following assertions:
\begin{itemize}
\item[\rm(i)] $\vp$ does not have finite angular
derivative at any point of $\T$;
\item[\rm(ii)] $\displaystyle\lim_{|z|\to1}\frac{\widehat{\om}(z)}{\widehat{\om}(\vp(z))}=0$;
\item[\rm(iii)] $\displaystyle\lim_{|z|\to1}\frac{\om^\star(z)}{\om^\star(\vp(z))}=0$.
\end{itemize}
Then {\rm(ii)}$\Rightarrow${\rm(i)}$\Leftrightarrow${\rm(iii)}. If in addition $\widehat{\om}(r)\ge C\widehat{\om}\left(\frac{1+r}{2}\right)$ for some $C=C(\om)>1$, then also \rm(i)$\Rightarrow$\rm(ii).
\end{lemma}

\section{Schatten class composition operators on Bergman spaces}

We begin with briefly discussing the Hilbert-Schmidt composition operators on~$A^2_\omega$. By \eqref{LP1}, $C_\vp\in\SSS_2(A^2_\omega)$ if and
only if
    \begin{equation*}
    \begin{split}
    \sum_{n=1}^\infty\left\|C_\vp\left(\frac{z^n}{\sqrt{2\om_n}}\right)\right\|^2_{A^2_\om}
    &\asymp\int_\D\left(\sum_{n=1}^\infty\frac{n^2|z|^{2n-2}}{\omega_n}\right)N_{\vp,\omega^\star}(z)\,dA(z)\asymp\int_\D\|(B_{z}^\omega)'\|_{A^2_\omega}^2N_{\vp,\omega^\star}(z)\,dA(z)<\infty.
    \end{split}
    \end{equation*}
Now \cite[Theorem~1]{PelRatproj} and Lemma~\ref{Lemma:replacement-Lemmas-Memoirs} imply
    \begin{equation*}
  \|(B_{z}^\omega)'\|_{A^2_\omega}^2
  \asymp\int_{0}^{|z|} \frac{dt}{\widehat{\om}(t)(1-t)^4}
  \asymp\frac{1}{\widehat{\om}(z)(1-|z|)^3}
  \asymp\frac{1}{(1-|z|)^2\omega^\star(z)},\quad|z|\to 1^-,
    \end{equation*}
and hence $C_\vp\in\SSS_2(A^2_\omega)$ if and
only if
    \begin{equation}\label{Eq:Hilbert-Schmidt2}
    \int_\D\frac{N_{\vp,\omega^\star}(z)}{\omega^\star(z)}\frac{dA(z)}{(1-|z|)^2}<\infty.
    \end{equation}
Theorem~\ref{Thm:intro-SchattenMain} states that composition operators in Schatten
$p$-classes are neatly characterized by a tidy condition
similar to \eqref{Eq:Hilbert-Schmidt2}. The purpose of this
section is to prove this theorem. The argument relies on Theorem~\ref{th:tmuextended}, Lemma~\ref{Nsubharmonic} and the following auxiliary result which allows us to discretice the condition \eqref{37}.

\begin{lemma}\label{le:equiv}
Let $0<p<\infty$, $\om\in\DD$ and $u$ a positive and subharmonic function on $\D$. Then
\begin{equation}\label{37u}
  \int_\D\left(\frac{u(z)}{\omega^\star(z)}\right)^p\frac{dA(z)}{(1-|z|)^2}<\infty
    \end{equation}
if and only if
\begin{equation}\label{37dis}
    \sum_{R_j\in\Upsilon}\left(\frac{\int_{R_j}u(z)\,dA(z)}{\omega^\star(z_j)(1-|z_j|)^2}\right)^p<\infty.
    \end{equation}
\end{lemma}

\begin{proof} Let $\{a_j\}_{j=0}^\infty$ be a fixed
$\delta$-lattice. Since $\om\in\DD$, a straightforward calculation shows that
\eqref{37dis} is equivalent to
    \begin{equation}\label{37dis2}
    \sum_{j=1}^\infty\left(\frac{\int_{\Delta(a_j,r)}u(z)\,dA(z)}{\omega^\star(a_j)(1-|a_j|)^2}\right)^p<\infty
    \end{equation}
for any $5\delta\le r<1$. Therefore it suffices to show that
\eqref{37u} and \eqref{37dis2} are equivalent.

Assume first that \eqref{37dis2} is satisfied, and let $0<s<1$ be
fixed. Then, by Lemma~\ref{Lemma:replacement-Lemmas-Memoirs},
    \begin{equation*}
    \begin{split}
    \int_\D\left(\frac{u(z)}{\omega^\star(z)}\right)^p\frac{dA(z)}{(1-|z|)^2}
    &\lesssim
    \int_\D\left(\int_{\Delta(z,s)}\frac{u(\z)}{(1-|\z|)^2\omega^\star(\z)}\,dA(\z)\right)^p\frac{dA(z)}{(1-|z|)^2}\\
    &\lesssim\sum_{j=1}^\infty\int_{\Delta(a_j,5\delta)}
    \left(\int_{\Delta(z,s)}\frac{u(\z)}{(1-|\z|)^2\omega^\star(\z)}\,dA(\z)\right)^p
    \frac{dA(z)}{(1-|z|)^2}\\
    &\lesssim\sum_{j=1}^\infty
    \left(\int_{\Delta(a_j,r)}\frac{u(\z)}{(1-|\z|)^2\omega^\star(\z)}\,dA(\z)\right)^p
    \lesssim\sum_{j=1}^\infty \left(\frac{\int_{\Delta(a_j,r)}u(\z)\,dA(\z)}
    {(1-|a_j|)^2\omega^\star(a_j)} \right)^p,
    \end{split}
    \end{equation*}
where $r=r(\d,s)\in(5\d,1)$, and thus \eqref{37u} is satisfied.

Conversely, assume \eqref{37u}, and let $r\in[5\d,1)$ be given.
Further, let $\tilde{a}_j\in\overline{\Delta(a_j,r)}$ such that $\sup_{z\in\Delta(a_j,r)}u(z)=u(\tilde{a}_j)$,
and let $0<s<1$ be fixed. Then, by using Lemma~\ref{Lemma:replacement-Lemmas-Memoirs},
a classical result of
Hardy-Littlewood~\cite{HLCrelle32} (see also
\cite[Lemma~3]{LuZhu92}), and known properties of a
$\delta$-lattice, we deduce
    \begin{equation*}
    \begin{split}
    \sum_{j=1}^\infty \left(\frac{\int_{\Delta(a_j,r)}
    u(\z)\,dA(\z)}{(1-|a_j|)^2\omega^\star(a_j)}\right)^p
    &\lesssim \sum_{j=1}^\infty\left(\frac{u(\tilde{a_j})}{\omega^\star(a_j)}\right)^p\lesssim \sum_{j=1}^\infty \int_{\Delta(\tilde{a}_j,s)}\left(\frac{u(z)}{\omega^\star(z)}\right)^p\frac{dA(z)}{(1-|z|)^2}\\
    &\lesssim \sum_{j=1}^\infty \int_{\Delta(a_j,t)}\left(\frac{u(z)}{\omega^\star(z)}\right)^p\frac{dA(z)}{(1-|z|)^2}
    \lesssim
    \int_\D\left(\frac{u(z)}{\omega^\star(z)}\right)^p\frac{dA(z)}{(1-|z|)^2},
    \end{split}
    \end{equation*}
where $t=t(r,s)\in(5\d,1)$.
\end{proof}

\medskip\noindent\emph{Proof of} Theorem~\ref{Thm:intro-SchattenMain}. The proof
uses arguments similar to those in \cite{LuZhu92}. However, it is simpler because $N_{\varphi,\om^\star}$ is subharmonic in $\D\setminus\{\vp(0)\}$
by Lemma~\ref{Nsubharmonic}, and hence a weighted version of
\cite[Lemma~1, p.~1132]{LuZhu92} is not needed.

Every automorphism $\vp_a$ induces a bounded invertible composition operator $C_{\vp_a}$ on $A^2_\om$. So, if  $\vp$ is an analytic self-map of $\D$, then $C_\vp C_{\vp_{\vp(0)}}=C_\sigma$, where $\sigma=\vp_{\vp(0)}\circ\vp$ satisfies $\sigma(0)=0$. Thus the operator $C_\vp\in \SSS_p(A^2_\omega)$ if and only if $C_\sigma$ does. Therefore throughout the proof we may assume that $\vp$ fixes the origin.

The differentiation $D(f)=f'$ is a bounded invertible
operator from the orthogonal complement of constants in
$A^2_\omega$ to $A^2_{\omega^\star}$. A straightforward
calculation shows that $T=DC_\vp D^{-1}$ satisfies
    \begin{equation}\label{38}
   \langle T^\star T(f),g\rangle_{A^2_{\omega^\star}}= \langle T(f),T(g)\rangle_{A^2_{\omega^\star}}=\int_\D
    f(z)\overline{g(z)}N_{\vp,\omega^\star}(z)\,dA(z).
    \end{equation}

Denote $d\mu(z)=N_{\vp,\omega^\star}(z)\,dA(z)$.
By \eqref{38} and \eqref{eq:st12}, we have $T^\star T=\mathcal{T}_\mu$, and so
$C_{\vp}\in\SSS_p(A^2_\omega)$ if and only if
$\mathcal{T}_\mu\in\SSS_{p/2}(A^2_{\om^\star})$, which is in turn
equivalent to
    $$
    \sum_{R_j\in\Upsilon}\left(\frac{\int_{R_j}N_{\vp,\omega^\star}(z)\,dA(z)}{\omega^\star(z_j)(1-|z_j|)^2}\right)^\frac{p}{2}<\infty
    $$
by Theorem~\ref{th:tmuextended}. The assertion follows by Lemmas~\ref{Nsubharmonic} and
\ref{le:equiv}.

We also provide a different proof of the fact that $C_\vp\in\SSS_p(A^2_\omega)$ implies \eqref{37} when $0<p<2$.
Assume, without loss of generality, that $\vp$
is non-constant.
Let
    $$
    (T^\star T)(f)=\sum_{n}\lambda_n\langle f,e_n^\star\rangle_{A^2_{\omega^\star}}e_n^\star,\quad f\in A^2_{\om^\star},
    $$
be the canonical decomposition of $T^\star T$. Then
$\{e_n^\star\}$ is an orthonormal basis. Indeed, if there were a
unit vector $e^\star\in A^2_{\omega^\star}$ such that
$e^\star\perp e^\star_ n$ for all $n\geq 1$, then
    \begin{displaymath}
    \int_{\D}|e^\star(z)|^2 N_{\vp,\om^\star}(z) \,dA(z)
    \asymp\|T(e^\star)\|_{A^2_{\omega^\star}}^2=\langle T^\star T(e^\star),e^\star\rangle_{A^2_{\omega^\star}} =0
    \end{displaymath}
because $T^\star T$ is a linear combination of the vectors
$e^\star_n$. Therefore we would have $e^\star\equiv0$ which is
obviously a contradiction. Now, \cite[Theorem~1]{PelRatproj} and Lemma~\ref{le:sc1} together with \eqref{RKformula} and  H\"older's
inequality yield
    \begin{eqnarray*}
    \int_\D\left(\frac{N_{\vp,\omega^\star}(z)}{\omega^\star(z)}\right)^\frac{p}{2}\frac{dA(z)}{(1-|z|)^2}
    &\asymp&\int_\D\left(\frac{N_{\vp,\omega^\star}(z)}{\omega^\star(z)}\right)^\frac{p}{2}
    \|B_z^{\omega^\star}\|^2_{A^2_{\omega^\star}}\omega^\star(z)\,dA(z)\\
    &\asymp&\sum_n\int_\D\left(\frac{N_{\vp,\omega^\star}(z)}{\omega^\star(z)}\right)^\frac{p}{2}
    |e_n^\star(z)|^2\omega^\star(z)\,dA(z)\\
    &\le&\sum_n\left(\int_\D N_{\vp,\omega^\star}(z)|e_n^\star(z)|^2\,dA(z)\right)^\frac{p}2\\
    &\asymp&\sum_n\langle(T^\star
    T)(e_n^\star),e_n^\star\rangle_{A^2_{\omega^\star}}^\frac{p}{2}.
    \end{eqnarray*}
This shows that \eqref{37} is satisfied whenever
$C_\vp\in\SSS_p(A^2_\omega)$.
An analogous argument gives that  \eqref{37} implies $C_\vp\in\SSS_p(A^2_\omega)$, when $2<p<\infty$.
\hfill$\Box$

\subsection*{Acknowledgements}
We wish to thank Manuel D.~Contreras and Wayne Smith for several discussions on composition operators.


\end{document}